\documentclass[11pt]{amsart}
\usepackage{amssymb}
\usepackage[breaklinks=true]{hyperref}
\usepackage[all, knot]{xy} 
\usepackage{bm}

\usepackage{mathrsfs}
\usepackage{graphicx}
\usepackage{subfigure}
\usepackage{flafter}

\usepackage{amsmath,amscd}
\newtheorem{theorem}{Theorem}
\newtheorem{Cor}{Corollary}

\usepackage{geometry}                
\newtheorem{remark}{Remark}
\newtheorem{old thm}{Theorem we need}
\newtheorem{lemma}[theorem]{Lemma}
\newtheorem{step}{Step}

\newcommand{\x}{\textbf{X}}

\newcommand{\h}{Hilb_m}

\begin{document}
\begin{center}
\textbf{The\;Full\;Flag\; Hilbert\; Scheme\; of\; Nodal\; Curves \;and the Punctual Hilbert Scheme of Points of the Cusp Curve. \; }\\[1.5cm]
 \textsc{HWAYOUNG LEE}\\[1cm]

\end{center}

\emph{\small ABSTRACT.} \textsc{\small We study the relative full-flag Hilbert scheme of a family of curves, parameterizing chains of subschemes,  containing a node. We will prove that the relative full flag Hilbert scheme is normal with locally complete intersection singularities.  We also study the Hilbert scheme of points of the cusp curve and show the punctual Hilbert scheme is isomorphic to $\mathbb{P}^1.$  We will see the Hilbert scheme has only one singularty along the punctual one.}  
\section{Introduction}

The Hilbert Scheme parameterizes ideal sheaves or subschemes of projective space or more generally, or a fixed scheme \x.
\begin{theorem}[Grothendieck]
The Hilbert functor is representable for an open subset of projective space.
\end{theorem}
\noindent The Hilbert  functor $$Hilb^{X}_{P(t)}:(Schemes)\rightarrow (Sets) $$ is defined by $Hilb^{\x}_{P(t)}(S)=\{$flat families  $\;\mathcal{X}\subset \x \times S$ of closed subschemes of \x \;parameterized by S with fibres having Hilbert polynomial $P(t)\}$. The theorem says that there exists a scheme Y representing the functor $Hilb^{\x}_{P(t)}$.  We call this scheme Y a $\text{Hilb}ert\; scheme \; of\; \x\\$ $\;relative\; to\;P(t)$ which is denoted by $\text{Hilb}_{P(t)}(\x)$. In other words, there is a universal family $\mathcal{W} \subseteq \text{Hilb}_{P(t)}(\x)\times \x$: 
$$\left.\begin{array}{ccc}\mathcal{W}& \subseteq & \text{Hilb}_{P(t)}(\x)\times \x \\ \pi\biggl{\downarrow}  &  &\\   \text{Hilb}_{P(t)}(\x)&  & \end{array}\right.$$
\noindent such that for each flat family T of closed subschemes  of $X\times S$ over a scheme S having Hilbert polynomial P(t): $$ \left.\begin{array}{ccc}T & \subseteq & \x\times S \\ \biggl{\downarrow} &  &  \\ S &  & \end{array}\right.$$is the pullback of $\mathcal{W}$ by a unique morphism $\phi$  ($i.e. \;T=\phi^{*}\mathcal{W}$) in the commutative diagram $$\left.\begin{array}{ccc}T & \longrightarrow & \mathcal{W}\\ \biggl{\downarrow} &  &\biggl{\downarrow} \\ S & \longrightarrow_{\phi} & \text{Hilb}_{P(t)}(\x).\end{array}\right.$$   For the \text{Hilb}ert polynomial constant m,  the \text{Hilb}ert scheme denoted by $\text{Hilb}_m(X) $\;parameterizes m points of \x\; or the set of ideals of colength m of $O_{\x}.$  The \text{Hilb}-Chow morphism  is defined by $$ Ch:\text{Hilb}_m(\x)\rightarrow \text{Sym}^m(\x) : Z\rightarrow \sum_{p\in\x}length_p(Z)[p].$$  Since the fibre at the point of a morphism is a scheme,  $Ch^{-1}(mp)=\text{Hilb}_m(O_{\x,p})$  is a scheme parameterizing m points supported only at the point $p\in \x.$   We want to see a local scheme structure at each point $\mathcal{I}\in \text{Hilb}_m(O_{\x,p})$.  Let us consider some open set U of \x containing p with $O_\x(U)=R$ and $O_{\x,p}=R_p$.  Let $\mathcal{A}$ be the category of artinian rings with residue field k. Consider functors of artinian rings
 $$\textrm{Def}_{I,R}, \textrm{Def}_{J,R_p} :\mathcal{A} \rightarrow (Sets)$$ defined by  $\;\textrm{Def}_{I,R}(S)=\{ideal\; I_s<R \otimes_k S|  R\otimes S/I_s\simeq\oplus^m S$\;and\;$I_s\otimes_Sk=I \}$ and $\textrm{Def}_{J,R_p}(S)=\{ideal J_s<\; R_p\otimes_k S| R_p\otimes_k S/J_s=\oplus ^mS$\; and $J_s\otimes_Sk=J \}$ for any $S\in \mathcal{A}.$  For any artinian local ring S with residue field k, there is one to one correspondence between $\textrm{Def}_{\phi^{-1}N,R}(S)$ and $\textrm{Def}_{N,R_p}(S)$ as a set\cite{S} where  $\phi: R\rightarrow R_p$  and an ideal $N$ in $R_p$. \;Since the functor $\textrm{Def}_{I,R}$ is prorepresentable (\cite{S}, p125), so is  $\textrm{Def}_{J,R_p}$. Every functor of artinian local rings with residue field k can be extended to a functor of  the category of complete local rings with residue field k denoted by  $\overline{\mathcal{A}}$. Thus if one want to see a local scheme structure at the point $I\in  \text{Hilb}_m(O_{\x},p)$,  it is enough to deform it in an artinian local rings with a residue field k. In a similar way, the relative \text{Hilb}ert functor of \x/B with \text{Hilb}ert polynomial m  is defined by  $\text{Hilb}^{\x/B}_m(S)=\{closed\; subschemes \;V\subset \x\times_BS\; which\; are\; proper\; and\; flat\; over \;S \;with\; the \;\text{Hilb}ert \;polynomial\; m\}$\; for any scheme S.  This is representable, the scheme representing it is denoted by $\text{Hilb}_m(\x/B).$ 
  We can show that $\text{Hilb}_m(X/B)\subset \text{Hilb}_m(X)$ is a closed scheme.[2]\; Let us consider the fibre of the composition of morphisms as below
$$ \text{Hilb}_m(\x/B)\subset \text{Hilb}_m(\x)\rightarrow Sym^m(\x),$$
then  we can get  $p^{-1}(mp)\cap \text{Hilb}_m(\x/B)$ is a scheme.\\



\section{The full flag Hilbert scheme of points of nodal curves.}

For the family of  0 dimensional subschemes of \x, \;we can consider the Hilb-Chow morphism
\begin{displaymath}
Ch:Hilb_m(\x)\rightarrow Sym^m(\x) : Z\rightarrow \sum_{p\in\x}length_p(Z)[p].
\end{displaymath}

 When \x\; is a smooth curve, $Hilb_m(\x)$ is coincide with  $Sym^m(\x)$. The next simplest case is the nodal curve \x.   To understand $Hilb_m(\x)$,\; it is important to understand the fibre of  $Ch$. When \x\; is a nodal curve, Z. Ran  studied the punctual Hilbert scheme $Hilb_m(O_\x,p)=Ch^{-1}(mp)$ \;for a nodal point p.\cite{R1}  (If you want to see more results about  the Hilb-Chow morphism for the nordal curve , you can see \cite{R2}).  The $punctual\; Hilbert \;scheme$\; $Hilb_m^0(X)$, parameterizing length-m subschemes supported at a node, is a chain  of m-1 $\mathbb{P}^1.$\; The relative Hilbert scheme of \x \;relative to B  denoted by  $Hilb_m(\tilde{\x}/B)$, parameterizing length m-subschemes in the fibre of  a family of $\tilde{X}/ B$ that is the map $\mathbb{A}^2\rightarrow \mathbb{A}^1$ 
given by $xy = t. $\; 
Let $m_{\bullet}=(m,m-1,\ldots,m-k)$ be the sequence of natural numbers of decreasing by one. Then the flag Hilbert scheme $Hilb_{m_{\bullet}}(\x)$ parameterizes the chains of $m_{\bullet}$-subschemes of \x.    The relative flag Hilbert scheme  $Hilb_{m_{\bullet}}(\tilde{\x}/B)$, parameterizing chains of $m_{\bullet}$-subschemes in the fibre of a family of $\tilde{\x}/B$.
$Hilb_{m,m-1}(\tilde{\x}/B)$ is smooth and  $Hilb_{m,m-1,m-2}(\tilde{\x}/B)$ is normal and locally complete intersection, but generally singular.\cite{R1} 
  We will show that  the relative full flag Hilbert scheme denoted by $fHilb_{m_{\bullet}}(\tilde{X}/B)$  for \;$m_{\bullet}=(m,m-1,\ldots,m-k)$ is normal and locally complete intersection as well.
In section 2, we  will review the results for the Hilbert scheme of points of the nodal curve in \cite{R1}, but more focus on computations we will use.   In section 3, we will start from the case $m_{\bullet}=(m, m-1, m-2, m-3)$ and generalized it for the case $m_{\bullet}=(m>m-1>\ldots>m-k)$. Then, we will see that  the relative full flag Hilbert scheme  is normal,  locally complete intersection and generally singular.


\subsection{Results on the Hilbert scheme of points in the nodal curve(in [R])}

This section contains known results that are relevant to our work.
Let $R=\mathbb{C}[[x,y]]/(xy)$ and note that R is isomorphic to the completion of $\mathbb{C}[x,y]/(xy)_{(x,y)}$ with maximal ideal $(x,y).$ 
\begin{theorem}(Ran, \cite{R1})
(i) Every ideal $I<R$ of colength m is of one of the following, said to be of type
 $(c_i^m),( q_i^m)$,respectively: $$I_i^m(a)=(y^i+ax^{m-i}),0\neq a\in \mathbb{C}, i=1,\ldots, m-1$$
 $$Q_i^m=(x^{m-i+1},y^i),i=1,\ldots,m$$
(ii)The punctual Hilbert scheme $Hilb^0_m(R)$, as algebraic set, is a rational chain 
 $$C^m_1\cup_{Q_2^m} C^m_2\cup\ldots\cup_{Q_{m-1}^m}C^m_{m-1} $$
\end{theorem} 
\begin{theorem}(Ran,\cite{R1})
The Hilbert scheme $Hilb_m(R)$
$$D_0^m\cup D_1^m\ldots\cup D_{m-1}^m\cup D_m^m$$
where each $D_i^m$ is a smooth and m-dimensional germ supported on
$C_i^m$ for $i=1,\ldots, m-1$ or $Q_i^m$ for $i=0,m$; for $i=1,\ldots, m-1, D_i^m$ meets its neighbors $D_{i\pm 1} ^m$ transversely in dimension $m-1$ and meets no other $D_i^m$. The generic point of $D_i^m$ corresponds to subscheme of $Spec(R)$ comprised of $m-i$ points on the x-axis and i points on the y-axis.
\end{theorem}
\begin{proof}
Clearly $\h(R)$ is a germ supported on $\h^0(R)$, so this is a matter of determining the scheme structure of $\h(R)$ at each point of $\h^0(R)$, which may be done formally by testing on Artin local algebras.
Given S artinian local $\mathbb{C}$- algebra, a flat S-deformation of $I=Q_i^m=(x^{m+1-i},y^i)$ is given by an ideal
$$I_s=(f,g),$$
\item $$f=x^{m+1-i}+f_1(x)+f_2(y)$$
$$g=y^i+g_1(x)+g_2(y),$$
where $f_i, g_i$ have coefficients in $m_s$ and $R_s/I_s$ is S-free of rank m.  By Nakayama's Lemma and a definition of a flat S-deformation
$$1,x,\ldots,x^{m-i},y,..y^{i-1}$$
is S-free basis for $R_s/I_s$ \\
We may assume that deg $f_1,g_1\leqslant m-i$ and $f_2,g_2< i$.
\item$$f_1(x)=\sum^{m-i}_{j=0}a_jx^j, f_2(y)=\sum ^{i-1}_{j=1}b_jy^j$$
$$g_1(x)=\sum_{j=0}^{m-i}c_jx^j, g_2(y)=\sum_{j=1}^{i-1}d_jy^j.$$
Now obviously
$$yf-b_{i-1}g\equiv0 \equiv xg-c_{m-i}f  \mod I_s$$
Since $yf-b_{i-1}g$ and $xg-c$ have terms only $1,x,\ldots, x^{m-i},y,\ldots,y^{i-1}$.
All coefficients of those are zeros, thus $yf-b_{i-1}g=0=xg-c$ in $R_s$.
Thus we get 
$$ b_j=b_{i-1}d_{j+1},j=1,\ldots,i-2,$$
$$b_{i-1}d_1=a_0,$$
\begin{align}\label{eq1}b_{i-1}c_j=0, j=0\ldots, m-i,\end{align}
$$c_j=c_{m-i}a_{j+1}, j=0,\ldots, m-i-1,$$
$$c_{m-i}a_0=0,$$
$$c_{m-i}b_j=0, j=1,\ldots,i-1.  $$

Conversely, suppose the relations  \eqref{eq1}  are satisfied, or equivalently 
$$yf-b_{i-1}g=0=xg-c_{m-i}f.$$
By Nakayama's Lemma, $1,x,\ldots, x^{m-i},y,\ldots,y^{i-1}$ generate $R_s/I_s$, hence to show $I_s$ defines a flat family it suffices to show these elements admit no nontrivial S-relations mod $I_s$. Suppose
  $$u_{m-i}(x)+v_{i-1}(y)=A(x,y)f+B(x,y)g\equiv 0$$
  where u,v, A, B are all polynimials with coefficients in S.
  Using S is artinian (D.C.C condition) and the relations  \eqref{eq1} , we get
 $$ A'=B'=u_{m-i}=v_{i-1}=0.$$  
Hence there are no nontrivial S-relations, as claimed.\\
Thus the Hilbert scheme is embedded in the space of the variables
 $$a_1,..a_{m-i},d_1,\ldots, d_{i-1},b_{i-1}c_{m-i}, $$
 i.e. $\mathbb{A}^{m+1}_{\mathbb{C}}$, and defined by the relation
 $$b_{i-1}c_{m-i}=0.$$
Thus it is a union of 2 smooth m-dimensional components meeting transversely in a smooth (m-1)-dimensional subvariety.
The generic point on the component where $b_{i-1}=0(resp. c_{m-i}=0)$ is clearly an ideal generated by g(resp. f), which has the properties as claimed.\\
\end{proof}
\begin{remark}
$a_1=\ldots a_{m-i}=d_1=\ldots d_{i-1}=0$ for $Hilb^0_m(R)\subseteq Hilb_m(R)$.
\end{remark}
\noindent
Next we consider the relative local situation, i.e. that of a germ of a family of curves with smooth total space specializing to a node. Thus set
  $$\tilde{R}=\mathbb{C}[x,y]_{(x,y)}, B=\mathbb{C}[t]_{(t)}$$
  and view $\tilde{R}$ as a B-module via $xy=t$. This is the versal deformation of the node singularity $xy=0$, so any family of nodal curves is locally a pullback of this.
By universal property of Hilbert scheme, for any
C-algebra S, we have a bijection between diagrams
\begin{displaymath}
 \left.\begin{array}{ccc}Spec S & \longrightarrow & Hilb_m(\tilde{R}/B) \\ &  & \biggl{\downarrow} \\ &  & Spec B\end{array}\right.
\end{displaymath}
and
\begin{displaymath}
\left.\begin{array}{ccc}I_s<\tilde{R}\otimes_B S & \longleftarrow& \tilde{R} \\\biggl{\uparrow} &  & \biggl{\uparrow} \\ S & \longleftarrow & B\end{array}\right.\end{displaymath}
 with $B\rightarrow S$ is a local homomorphism defined by $t\rightarrow s\in m_s$ for some s.
\begin{theorem}[Ran, \cite{R1}]
The relative Hilbert scheme $Hilb_m(\tilde{R}/B)$ is formally smooth, formally (m+1)-dimensional over $\mathbb{C}.$
\end{theorem}
\begin{proof}
The relative Hilbert scheme parametrizes length-m schemes contained in fibres of $Spec\tilde{R}\rightarrow Spec B$. This means ideals $I_s<\tilde{R}_s$ if colength m containing $xy-s$ for some $s\in m_s$, such that $\tilde{R}_s/I_s$ is S-free. The analysis of these is virtually identical to that contained in the proof of Theorem 2, except that the relation $b_{i-1}c_{m-i}=0$ gets replaced by
$$b_{i-1}c_{m-i}=s$$
and lines 3,5,6 of display \eqref{eq1} are replaced, respectively, by
\begin{align}\label{eq2}b_{i-1}c_j=sa_{j+1},j=0,\ldots,m-i-1\end{align}
$$c_{m-i}a_0=sd_1$$
$$c_{m-i}b_j=sd_{j+1},1\leqslant j\leqslant i-2$$
relations which already follow from the other relations (in lines 1,2,4, of display (7)) combined with the relation $b_{i-1}c_{m-i}=s$. Thus, the relative Hilbert scheme is the subscheme of the affine space of the variables $a_1,\ldots, a_{m-i},d_1,\ldots, d_{i-1},b_{i-1}c_{m-i},t$ defined by the relation
$$b_{i-1}c_{m-i}=t$$
hence is smooth as claimed.
\end{proof}
\noindent
Now we are discussing about the flag Hilbert schemes.
For any decreasing sequences of positive integers
$$m_{\bullet}=(m_1 \ldots m_k)$$
the flag Hilbert scheme $Hilb_{m_{\bullet}}(R)$ parametrizes nested chains of ideals
$$I_s^1\subset \ldots\subset I_s^k \subset R_s$$
such that $R_s/I_s^j$ is S-free of rank $m_j,j=1,\ldots, k$.
(For $Hilb_{m_{\bullet}}(\tilde{R}/B)$, the chain of ideals containing $xy-s$ for some $s\in m_s$ with $R_s/m_s$ is S-free of rank m.
\begin{theorem}(Ran,\cite{R1})
(i)The punctual flag Hilbert scheme $Hilb^0_{m,m-1}(R)$, as algebraic set, is a chain of nonsingular rational curves of the form
$$C^m_1\cup_{(Q_2^m,Q_1^m)}C^{m-1}_1\cup_{(Q_2^m,Q_2^{m-1})}C_2^m\cup\ldots\cup_{(Q^m_{m-1},Q^{m-1}_{m-1})}C^m_{m-1} $$
it has ordinary nodes at $Q^m_2,\ldots, Q^m_{m-1}$ and is smooth elsewhere. Each component $C^m_i$ projects isomorphically to its image in $Hilb^0_m(R)$ and to a point $Q^m_i$ in $Hilb_{m-1}(R)$, and vice-versa for $C^{m-1}_i.$\\
(ii)The flag Hilbert scheme $Hilb_{m,m-1}(R)$, as formal scheme along $Hilb_{m,m-1}$, has normal crossing singularities and at most triple points. Each of its components is formally smooth, m-dimensional and of the form
$$D^{m,m-1}_{i,i'},i=0,\ldots,m, i-1\leqslant i'\leqslant i$$
(iii) The relative flag Hilbert scheme $Hilb_{m,m-1}(\tilde{R}/B)$, as formal scheme along $Hilb^0_{m,m-1}(R)$, is formally smooth and (m+1)-dimensional over $\mathbb{C}$. The natural map $$Hilb_{m,m-1}(\tilde{R}/B)\rightarrow Hilb_{m-1}(\tilde{R}/B)  $$
is a flat, locally complete intersection morphism of relative dimension 1.
\end{theorem}
\begin{proof}
Too see a scheme structure of $Hilb_{m,m-1}(R)$ at  $(Q^m_i,Q^{m-1}_i),1<i<m$.
 $\{(I_s<I_s')|I_s,I_s'<R\otimes_\mathbb{C} S, dim_s(R_s/I_s)=m, dim_s(R_s/I_s)=m-1, I_s, I_s'$ are flat deformation of $I,I' resp.\}$\\ with $I=(x^{m+1-i},y^i), I'=(x^{m-i},y^i).$
To see a scheme structure of $Hilb_{m,m-1}(R/B)$,
 for some $s\in m_s$, $I_s$ and $I_s'$ generated by $xy-s$ and $f, g (resp. f',g')$ with
$$f=x^{m+1-i}+\sum^{m-i}_{j=0}a_jx^j+\sum ^{i-1}_{j=1}b_jy^j,\;g=y^i+\sum_{j=0}^{m-i}c_jx^j+\sum_{j=1}^{i-1}d_jy^j.$$
$$f'=x^{m-i}+\sum^{m-i-1}_{j=0}a_j'x^j+\sum ^{i-1}_{j=1}b_j'y^j,\;g'=y^i+\sum_{j=0}^{m-i-1}c_j'x^j+\sum_{j=1}^{i-1}d_j'y^j.$$\\
As we say above, the relations \eqref{eq2} are necessary and sufficient so that $I_s,\;I_s'$ are S-flat deformations of $I,\;I'$  respectively.
In particular, these include the equations:
$$b_{i-1}c_{m-i}=s=b_{i-1}'c_{m-i-1}'.$$
The other relations remained can be from $I_s<I_s'$. 
To this end it suffices to note that
$$ 1,x,x^2,\ldots, x^{m-i-1},y,\ldots,y^{i-1}$$ form an S-free basis of $R_s/I_s'$, then express $f, g$ in terms of this basis and equate the coefficients to 0. 
Thus
$$a_0=-(a_{m-i}-a_{m-i-1}')a_0'+sb_1'$$
$$a_j=a_{j-1}'-(a_{m-i}-a_{m-i-1}')a_j',j=1,\ldots,m-i-1$$
\begin{align}\label{eq3}b_j=(a_{m-i}-a_{m-i-1}')b_j'+sb_{j+1}',j=1,\ldots,i-2\end{align}
$$\mathbf{b_{i-1}=(a_{m-i}-a_{m-i-1}')b_{i-1}'}$$
$$c_j=c_j'+c_{m-i}a_j', j=0,\ldots,m-i-1$$
$$d_j=d_j'+c_{m-i}b_j', j=1,\ldots,i-1$$\\
These coefficient relations are equivalent to 
$$f=(x+a_{m-i}-a_{m-i-1}')f', g=g'+c_{m-i}f'.$$
By formal manipulations, these relations imply that
$$\mathbf{c_{m-i-1}'=c_{m-i}(a_{m-i}-a_{m-i-1}')}.$$
Thus, We can get
$$(a_{m-i}-a_{m-i-1}')b_{i-1}'c_{m-i}=s.$$
Consequently,
 the relative flag Hilbert scheme is smooth here, with regular parameters
$$a_1',\ldots, a_{m-i-1}',a_{m-i},d_1',\ldots, d_{i-1}',b_{i-1}',c_{m-i}$$
and its fibre, i.e. $Hilb_{m,m-1}(R)$, it the normal crossing  triple point
$$(a_{m-i}-a_{m-i-1}')b_{i-1}'c_{m-i}=0.$$
The 3 components are: $D^{m,m-1}_{i-1,i-1}$ defined by $a_{m-i}-a_{m-i-1}'=0$;
$D^{m_{\bullet}m-1}_{i,i}$, defined by $b_{i-1}'=0$;$D^{m,m-1}_{i,i-1}$, defined by $c_{m-1}=0$. \\
Finally the relation [15] exhibits $Hilb_{m,m-1}(\tilde{R}/B)$ locally as a conic in an $\mathbb{A}^2$ with coordinates $a_{m-i},c_{m-i}$ over $Hilb_{m-1}(\tilde{R}/B)$, and therefore the projection is a flat locally complete intersection morphism.
\end{proof}
\begin{remark}(Ran,\cite{R1})\; Continuing Theorem above, for the further reference we note that in the analogue case of (relative) deformations of $(Q^m_i, Q^{m-1}_{i-1})\in Hilb_{m_{\bullet}}(\tilde{R}/B)$ \end{remark}
\begin{proof}
 A flat S-deformation of an ideal $I'=Q_{i-1}^{m-1}$ is given by a ideal $I_s'=(f',g')$ of $\tilde{R}_s$ relative to B,
\begin{center}
 $f'=x^{m-i+1}+\sum^{m-i}_{j=0}a_j'x^j+\sum^{i-2}_{j=1}b_j'y^j, \;g'=y^{i-1}+\sum^{m-i}_{j=0}c_j'x^j+\sum^{i-2}_{j=1}d_j'y^j.$ \end{center}
As we say above, the relations \eqref{eq2}  are necessary and sufficient so that $I_s,\;I_s'$ are S-flat deformations of $I,\;I'$ respectively.
In particular, these include the equations:
$$b_{i-1}c_{m-i}=s=b_{i-1}'c_{m-i-1}'.$$

The other relations remained can be from $I_s<I_s'$. 
To this end it suffices to note that
$$ 1,x,x^2,\ldots, x^{m-i},y,\ldots, y^{i-2}$$ form an S-free basis of $R_s/I_s'$, then express $f, g$ in terms of this basis and equate the coefficients to 0. 
Thus
$$a_k=a_k'+b_{i-1}c_k'\; for\;k=0,\ldots,m-i$$
$$b_k=b_k'+b_{i-1}d_k'\;\; for\;k=1,\ldots, i-2$$
\begin{align}\label{eq4}c_k=c_{k+1}'s+(d_{i-1}-d_{i-2}')c_k\; for\; k=0,\ldots,m-i-1\end{align}
$$c_{m-i}=(d_{i-1}-d_{i-2}')c_{m-i}'$$
$$d_1=c_0+(d_{i-1}-d_{i-1}')d_1$$
$$d_k=d_{k-1}'+(d_{i-1}-d_{i-2}')d_k'\; for\; k=2,\ldots,i-2$$
 or equivalently 
$$f=f'+b_{i-1}g',\; g=(y+d_{i-1}-d_{i-2}')g'$$
By formal manipulations, these relations imply that
$\mathbf{ b_{i-2}'=(d_{i-1}-d_{i-2}')b_{i-1}}$(using above equation and\; $b_{i-2}= b_{i-1}d_{i-1}$ in  \eqref{eq1} and \;$\mathbf{c_{m-i}=(d_{i-1}-d_{i-2}')c_{m-i}'}.$)
Hence we have the relation
$$(d_{i-1}-d_{i-2}')b_{i-1}c_{m-i}'=s.$$
Concequently, the relative flag Hilbert scheme$Hilb_{m,m-1}(R)$ at this pont has $m+1$ regular parameters
$$a_1',\ldots, a_{m-i}', d_1',\ldots,d_{i-2}',c_{m-i}',d_{i-1}, b_{i-1}$$
and its fibre, i.e. $Hilb_{m,m-1}(R)$, it the normal crossing  triple point
$$(d_{i-1}-d_{i-2}')b_{i-1}c_{m-i}'=0.$$
\end{proof}
\noindent
\begin{lemma} (Ran,\cite{R1})
Set  $m_{\bullet}=(m, m-1, m-2)$. Then (i) As algebraic set, $Hilb_{m_{\bullet}}(R)$ is of the form
 $$C^m_1\cup C^{m-1}_1\cup C^{m,m-2}_{2,1}\cup C^{m-1}_2\cup \ldots\cup C^{m,m-2}_{m-2,m-3}\cup C^{m-1}_{m-2}\cup C^{m}_{m-1}.$$
 Each component $C^{m,m-2}_{i,i-1}$ projects isomorphically to $C^m_i\times C^{m-2}_{i-1}\subset Hilb_{m,m-2}(R)$ and to $\{ Q^{m-1}_{i} \} \subset Hilb_{m-1}(R).$\\
(ii) $Hilb_{m_{\bullet}}(\tilde{R}/B)$ is irreducible and is smooth except at points $(Q_i^m,Q_i^{m-1}, Q^{m-2}_{i-1})$, where is has a rank 4 quadratic hypersurface singularity with local equation
$$(a_{m-i}-a_{m-i-1}')c_{m-i}=(d_{i-1}'-d_{i-2}'')c_{m-i-1}''.$$
\end{lemma}
\noindent
\begin{theorem}(Ran,\cite{R1}) The (full) flag Hilbert scheme $fHilb_m(\tilde{R}/B)$ has locally complete intersection singularties and its natural map to B is local complete intersection morphism. In particular $fHilb_m(\tilde{R}/B)$ is reduced and is flat over B
\end{theorem}
\noindent
\begin{Cor}(Ran,\cite{R1})The full flag punctual Hilbert scheme $fHilb_m^0(R)=Hilb_{m,..,1}^0(R)$ is reduced and is the transverse union of components of the form

$$C^{m,m-2,\ldots, m-2j}_{i,i-1,\ldots,i-j}, \forall i, 1\leqslant i\leqslant m-1, j=min(\left[\frac{m}{2}\right],i-1,m-i-1)\geqslant 0,$$ which projects isomorphically to $C^m_i\times\ldots\times C^{m-2j}_{i-j}$, and to a point in the other factors. The punctual Hilbert scheme $Hilb_m^0(R)$, with the scheme structure as above is a reduced nodal curve.
\end{Cor}

We will prove the statement in Theorem 7 above "the full flag Hilbert scheme has locally complete intersection singularties" using a downward induction in the  Lemma 6.(Ran, \cite{R1})

\subsection{The relative flag \text{Hilb}ert scheme of points for  $m_{\bullet}=(m,m-1,m-2,m-3)$}

\begin{lemma}
Set $m_{\bullet}=(m,m-1,m-2,m-3)$. Then\\
(1)$\text{Hilb}_{m_{\bullet}}(\tilde{R}/B)$ is smooth at points $(Q_i^m,Q_i^{m-1},Q_i^{m-2}, Q_{i}^{m-3})$\\
(2) $\text{Hilb}_{m_{\bullet}}(\tilde{R}/B)$ at points $(Q_i^m,Q_i^{m-1},Q_i^{m-2}, Q_{i-1}^{m-3})$ has a hypersurface singularity with local equation  $$c_{m-i}^0(a_{m-i}^0-a_{m-i-1}^1)(a_{m-i-1}^1-\left(a_{m-i-2}^3+b_{i-1}^2c_{m-i-2}^2\right))=(d_{i-1}^2-d_{i-2}^3)c_{m-i-2}^3.$$
\\
(3) $\text{Hilb}_{m_{\bullet}}(\tilde{R}/B)$ at points $(Q_i^m,Q_i^{m-1},Q_{i-1}^{m-2}, Q_{i-1}^{m-3})$ is a local complete intersection with local equations 
\begin{align}
&\left(d_{i-1}^1-\left(d_{i-2}^3+c_{m-i-1}^2b_{i-2}^3\right)\right)b_{i-1}^1=\left(a_{m-i-1}^2-a_{m-i-2}^3\right)b^3_{i-2},\notag\\
c^1_{m-i-1}= 
&(d_{i-1}^1-\left(d_{i-2}^3+c_{m-i-1}^2b_{i-2}^3\right))c_{m-i-1}^2=(a_{m-i}^0-\left(a_{m-i-1}^2+b_{i-1}^1c_{m-i-1}^2\right))c_{m-i}^0.\notag
\end{align}
(4)$\text{Hilb}_{m_{\bullet}}(\tilde{R}/B)$ at points $(Q_i^m,Q_{i-1}^{m-1},Q_{i-1}^{m-2}, Q_{i-1}^{m-3})$ has a hypersurface singularity with local equation  $$(a^1_{m-i}-a^2_{m-i-1})(a^2_{m-i-1}-a^3_{m-i-2})b^3_{i-2}=\left\{d^0_{i-1}-d_{i-2}^3-\left(a^1_{m-i}-a^3_{m-i-2}\right)c_{m-i}^1b_{i-2}^3\right\}b_{i-1}^0.$$
\end{lemma}
\begin{proof}
 For a point  $(Q_i^m, Q_i^{m-1}, Q_i^{m-2}, Q_{i}^{m-3}) \in \text{Hilb}^0_{m_\bullet}(R)$,  consider  a pair $(I_s^0<I_s^1<I_s^2<I_s^3)$  flatly deforming $(Q_i^m, Q_i^{m-1}, Q_i^{m-2}, Q_i^{m-3})$ relative to B. Then we may assume that for some $s \in m_s$,  \;$I_s^k$ is generated by $xy-s$ and $f^k, g^k\;$ for $k=0,1,2,3$ with $a_j^{k}, \;b_j^{k},\; c_j^{k},\;d_j^{k}\in m_s$ where
\begin{align}
\notag&f^0=x^{m-i+1}+\sum^{m-i}_{j=0}a_j^0x^j+\sum^{i-1}_{j=1}b_j^0y^j,\;g^0=y^{i}+\sum^{m-i}_{j=0}c_j^0x^j+\sum^{i-1}_{j=1}d_j^0y^j,\\ \notag
&f^1=x^{m-i}+\sum^{m-i-1}_{j=0}a_j^1x^j+\sum^{i-1}_{j=1}b_j^1y^j,\;g^1=y^{i}+\sum^{m-i-1}_{j=0}c_j^1x^j+\sum^{i-1}_{j=1}d_j^1y^j,\\ \notag
&f^2=x^{m-i-1}+\sum^{m-i-2}_{j=0}a_j^2x^j+\sum^{i-1}_{j=1}b_j^2y^j,\;g^2=y^{i-1}+\sum^{m-i-2}_{j=0}c_j^2x^j+\sum^{i-1}_{j=1}d_j^2y^j,\\ \notag
&f^3=x^{m-i-2}+\sum^{m-i-3}_{j=0}a_j^3x^j+\sum^{i-1}_{j=1}b_j^3y^j,\;g^3=y^{i}+\sum^{m-i-3}_{j=0}c_j^3x^j+\sum^{i-1}_{j=1}d_j^3y^j.\\ \notag
\end{align}
Since each $I_s^k$ is a flat ideal of $R_s$,  the equations  \eqref{eq2} all the local possible parameters for the $\text{Hilb}_{m_{\bullet}}(\tilde{R}/B)$ at  the point $(Q_i^m, Q_i^{m-1}, Q_i^{m-2}, Q_{i}^{m-3})$ are \begin{align}
\notag&a_1^0,\ldots,a_{m-i}^0, d_1^0,\ldots,d_{i-1}^0, b_{i-1}^0,c_{m-i}^0,
a_1^1,\ldots, a_{m-i-1}^1, d_1^1,\ldots,d_{i-1}^1, b_{i-1}^1,c_{m-i-1}^1,\\
\notag&a_1^2,\ldots, a_{m-i-2}^2, d_1^2,\ldots,d_{i-1}^2, b_{i-1}^2,c_{m-i-2}^2,a_1^3,\ldots, a_{m-i-3}^3, d_1^3,\ldots,d_{i-1}^3, b_{i-1}^3,c_{m-i-3}^3, s.
\end{align}
Then let us consider the nested condition between ideals.  Since we have a condition $I_s^0<I_s^1$,  the equations \eqref{eq3} and in particular
\begin{center}
$c^1_{m-i-1}=c_{m-i}^0(a_{m-i}-a^1_{m-i-1}),\; (a_{m-i}^0-a^1_{m-i-1})b^1_{i-1}c_{m-i}^0=s$
\end{center}
implies that the parameters above can be reduced to
\begin{center}
$a_{m-i}^0, c_{m-i}^0, a_1^1,\ldots, a_{m-i-1}^1, d_1^1,\ldots,d_{i-1}^1, b_{i-1}^1,$\\
$a_1^2,\ldots, a_{m-i-2}^2, d_1^2,\ldots,d_{i-1}^2, b_{i-1}^2,c_{m-i-2}^2,a_1^3,\ldots, a_{m-i-3}^3, d_1^3,\ldots,d_{i-1}^3, b_{i-1}^3,c_{m-i-3}^3, s.$
\end{center}
with $$(a_{m-i}^0-a^1_{m-i-1})b^1_{i-1}c^0_{m-i}=s.$$
Since we have $I_s^1<I_s^2$, the equations  \eqref{eq3}  and in particular
\begin{center}
$c^2_{m-i-2}=c^1_{m-i-1}(a^1_{m-i-1}-a^2_{m-i-2}),\; (a^1_{m-i-1}-a^2_{m-i-2})b^2_{i-1}c^1_{m-i-1}=s.$
\end{center}
implies that the parameters can be more reduced to
\begin{center}
$ a_{m-i}^0, c_{m-i}^0, a_{m-i-1}^1,$\\$a_1^2,\ldots,a_{m-i-2}^2, d_1^2,\ldots,d_{i-1}^2, b_{i-1}^2, a_1^3,\ldots, a_{m-i-3}^3, d_1^3,\ldots,d_{i-1}^3, b_{i-1}^3,c_{m-i-3}^3, s$
\end{center}
with  a equation \begin{center}
        $b^2_{i-1}c_{m-i}^0(a_{m-i}^0-a^1_{m-i-1})(a^1_{m-i-1}-a^2_{m-i-2})=s.$
        \end{center}
  Similarly,  $\;I_s^2<I_s^3$ implies the equations  \eqref{eq3} and in particular \begin{center}
$c^3_{m-i-3}=c^2_{m-i-2}(a^2_{m-i-2}-a^3_{m-i-3}),\; (a^2_{m-i-2}-a^3_{m-i-3})b^3_{i-1}c^2_{m-i-2}=s.$ \end{center}
Thus the parameters can be reduced to \begin{center}
$ a_{m-i}^0, c_{m-i}^0,a_{m-i-1}^1,a_{m-i-2}^2, a_1^3,\ldots, a_{m-i-3}^3, d_1^3,\ldots, d_{i-1}^3, b_{i-1}^3, s.$ \end{center}
with \begin{center}   $(a^2_{m-i-2}-a^3_{m-i-3})b^3_{i-1}c_{m-i}^0(a_{m-i}^0-a^1_{m-i-1})(a^1_{m-i-1}-a^2_{m-i-2})=s.$ \end{center}
Hence $\text{Hilb}_{m_{\bullet}}(\tilde{R}/B)$ at $(Q_i^m, Q_i^{m-1}, Q_i^{m-2}, Q_{i}^{m-3})$ is smooth with $m+1$ regular parameters.
The flag \text{Hilb}ert scheme $\text{Hilb}^0_{m_{\bullet}}(R)$ at $(Q_i^m, Q_i^{m-1}, Q_i^{m-2}, Q_{i}^{m-3})$ is embedded in $\mathbb{A}^{m+1}$ with parameters
\begin{center}
$ a_{m-i}^0, c_{m-i}^0, a_{m-i-1}^1,a_{m-i-2}^2, a_1^3,\ldots, a_{m-i-3}^3, d_1^3,\ldots, d_{i-1}^3, b_{i-1}^3.$\\
\end{center}
with   one equation  $$(a_{m-i}^0-a^1_{m-i-1})(a^1_{m-i-1}-a^2_{m-i-2})(a^2_{m-i-2}-a^3_{m-i-3})b^3_{i-1}c_{m-i}^0=0.$$
At the point $(Q_i^m, Q_i^{m-1}, Q_i^{m-2}, Q_{i-1}^{m-3})\in \text{Hilb}_{m_{\bullet}}^0(R)$, 
let us consider $(I_s^0<I_s^1<I_s^2<I_s^3)$ flatly deforming  of ideals$(Q_i^m,Q_i^{m-1},Q_i^{m-2},Q_{i-1}^{m-3})$ relative to B. Then we may assume for some  $s \in m_s$,  $I_s^k$ is generated by $xy-s$ and $f^k, g^k$ for $k=1,2,3,4.$ with $a_j^k, b_j^k, c_j^k, d_j^k\in m_s$ where 
\begin{align}
\notag&f^0=x^{m-i+1}+\sum^{m-i}_{j=0}a_j^0x^j+\sum^{i-1}_{j=1}b_j^0y^j,\;g^0=y^{i}+\sum^{m-i}_{j=0}c_j^0x^j+\sum^{i-1}_{j=1}d_j^0y^j\\\notag
&f^1=x^{m-i}+\sum^{m-i-1}_{j=0}a_j^1x^j+\sum^{i-1}_{j=1}b_j^1y^j,\;g^1=y^{i}+\sum^{m-i-1}_{j=0}c_j^1x^j+\sum^{i-1}_{j=1}d_j^1y^j\\\notag
&f^2=x^{m-i-1}+\sum^{m-i-2}_{j=0}a_j^2x^j+\sum^{i-1}_{j=1}b_j^2y^j,\;g^2=y^{i}+\sum^{m-i-2}_{j=0}c_j^2x^j+\sum^{i-1}_{j=1}d_j^2y^j\\\notag
&f^3=x^{m-i-1}+\sum^{m-i-2}_{j=0}a_j^3x^j+\sum^{i-2}_{j=1}b_j^3y^j,\;g^3=y^{i-1}+\sum^{m-i-2}_{j=0}c_j^3x^j+\sum^{i-2}_{j=1}d_j^3y^j.\notag
\end{align}

Since each $I_s^k$ is a flat ideal of $R_s$,  from the equations  \eqref{eq2},  variables for the $\text{Hilb}_{m_{\bullet}}(R)$ at  $(Q_i^m, Q_i^{m-1}, Q_i^{m-2}, Q_{i-1}^{m-3})$ are
\begin{align}
\notag& a_1^0,\ldots,a_{m-i}^0, d_1^0,\ldots,d_{i-1}^0, b_{i-1}^0,c_{m-i}^0, a_1^1,\ldots, a_{m-i-1}^1, d_1^1,\ldots,d_{i-1}^1, b_{i-1}^1,c_{m-i-1}^1,\\ \notag& a_1^2,\ldots, a_{m-i-2}^2, d_1^2,\ldots,d_{i-1}^2, b_{i-1}^2,c_{m-i-2}^2, a_1^3,\ldots, a_{m-i-2}^3, d_1^3,\ldots,d_{i-2}^3, b_{i-2}^3,c_{m-i-2}^3,s.
\end{align}
Since $I_s^0<I_s^1$, that gives equations  \eqref{eq3} and in particular
\begin{center}
$c^1_{m-i-1}=c^0_{m-i}(a_{m-i}^0-a^1_{m-i-1}),\;  b_{i-1}^{0}=(a_{m-i}^{0}-a_{m-i-1}^{1})b_{i-1}^{1}, (a_{m-i}^0-a^1_{m-i-1})b^1_{i-1}c^0_{m-i}=s.$
\end{center} gives
\begin{center}
$a_{m-i}^0, c_{m-i}^0, a_1^1,\ldots, a_{m-i-1}^1, d_1^1,\ldots,d_{i-1}^1, b_{i-1}^1,c_{m-i-1}^1,$\\
$a_1^2,\ldots, a_{m-i-2}^2, d_1^2,\ldots,d_{i-1}^2, b_{i-1}^2,c_{m-i-2}^2, a_1^3,\ldots, a_{m-i-2}^3, d_1^3,\ldots,d_{i-2}^3, b_{i-2}^3,c_{m-i-2}^3, s.$
\end{center}
with $(a_{m-i}^0-a^1_{m-i-1})b^1_{i-1}c^0_{m-i}=s.$
Since we have relations $I_s^1<I_s^2$ and  $I_s^2<I_s^3$,   we have equations with different indices in \eqref{eq3} and in particular
$$c^2_{m-i-2}=c^1_{m-i-1}(a^1_{m-i-1}-a^2_{m-i-2}),$$
$$ b_{i-1}^{1}=(a_{m-i-1}^{1}-a_{m-i-2}^{2})b_{i-1}^{2},\; (a^1_{m-i-1}-a^2_{m-i-2})b_{i-1}c^1_{m-i-1}=s$$

\begin{center}
$c^2_{m-i-2}=(d^2_{i-1}-d^3_{i-2})c^3_{m-i-2},\; b_{i-2}^3=(d_{i-1}^{2}-d_{i-2}^{3})b_{i-1}^2.$
\end{center}
Consequently, the relative four flag \text{Hilb}ert scheme at  the point $(Q_i^m, Q_i^{m-1}, Q_i^{m-2}, Q_{i-1}^{m-3})$,
\begin{center}
$a_{m-i}^0,c_{m-i}^0,a^1_{m-i-1}, d^2_{i-1},b^2_{i-1}, a_1^3,\ldots a_{m-i-3}^3,a_{m-i-2}^3, d_1^3,\ldots,d_{i-2}^3,c_{m-i-2}^3$
\end{center}
  with the relation
   \begin{center}
 $(c^2_{m-i-2}=)c_{m-i}^0(a_{m-i}^0-a_{m-i-1}^1)(a_{m-i-1}^1-(a_{m-i-2}^3+b_{i-1}^2c_{m-i-2}^2))=(d_{i-1}^2-d_{i-2}^3)c_{m-i-2}^3.$\end{center}
The Flag \text{Hilb}ert scheme $\text{Hilb}_{m_{\bullet}}(R)$ at the point $(Q_i^m, Q_i^{m-1}, Q_i^{m-2}, Q_{i-1}^{m-3})$  we have \begin{center}
$(a_{m-i}^0-a_{m-i-1}^1)(a_{m-i-1}^1-(a_{m-i-2}^3+b_{i-1}^2c_{m-i-2}^2))\mathbf{ b_{i-1}^2} c_{m-i}^0=(d_{i-1}^2-d_{i-2}^3)\mathbf{ b_{i-1}^2}c_{m-i-2}^3=0.$\end{center}

(3) At the point $(Q_i^m, Q_{i}^{m-1}, Q_{i-1}^{m-2}, Q_{i-1}^{m-3})\in \text{Hilb}_{m_\bullet}^0(R)$. 
Let $(I_s^0<I_s^1<I_s^2<I_s^3)$ be flatly deforming $Q_i^m, Q_{i}^{m-1}, Q_{i-1}^{m-2},Q_{i-1}^{m-3}$. Then we may assume that for some $s \in m_s$,  $I_s^k$ is generated by $xy-s$ and $f^k, g^k$ for $k=1,2,3,4.$
\begin{align}
\notag&f^0=x^{m-i+1}+\sum^{m-i}_{j=0}a_j^0x^j+\sum^{i-1}_{j=1}b_j^0y^j,\;g^0=y^{i}+\sum^{m-i}_{j=0}c_j^0x^j+\sum^{i-1}_{j=1}d_j^0y^j\\\notag
&f^1=x^{m-i}+\sum^{m-i-1}_{j=0}a_j^1x^j+\sum^{i-1}_{j=1}b_j^1y^j, \;g^1=y^{i}+\sum^{m-i-1}_{j=0}c_j^1x^j+\sum^{i-1}_{j=1}d_j^1y^j\\\notag
&f^2=x^{m-i}+\sum^{m-i-1}_{j=0}a_j^2x^j+\sum^{i-2}_{j=1}b_j^2y^j,\;g^2=y^{i-1}+\sum^{m-i-1}_{j=0}c_j^2x^j+\sum^{i-2}_{j=1}d_j^2y^j\\\notag
&f^3=x^{m-i-1}+\sum^{m-i-2}_{j=0}a_j^3x^j+\sum^{i-2}_{j=1}b_j^3y^j,\;g^3=y^{i-1}+\sum^{m-i-2}_{j=0}c_j^3x^j+\sum^{i-2}_{j=1}d_j^3y^j\notag
\end{align}
The relative four flag \text{Hilb}ert scheme at  the point $(Q_i^m, Q_i^{m-1}, Q_{i-1}^{m-2}, Q_{i-1}^{m-3})$,
\begin{align}
\notag&a_1^0,\ldots, a_{m-i}^0, d_1^0,\ldots, d_{i-1}^0, b_{i-1}^0,c_{m-i}^0,
a_1^1,\ldots, a_{m-i-1}^1, d_1^1,\ldots, d_{i-1}^1, b_{i-1}^1,c_{m-i-1}^1,\\
\notag&a_1^2,\ldots, a_{m-i-1}^2, d_1^2,\ldots, d_{i-2}^2, b_{i-2}^2,c_{m-i-1}^2,
a_1^3,\ldots, a_{m-i-2}^3, d_1^3,\ldots, d_{i-2}^3, b_{i-2}^3,c_{m-i-2}^3, s
\end{align}
Since $I_s<I_s^1$, by equations \eqref{eq3} and in particular
$$c^1_{m-i-1}=c_{m-i}^0(a_{m-i}^0-a^1_{m-i-1}),\; (a_{m-i}-a^1_{m-i-1})b^1_{i-1}c^0_{m-i}=s$$
gives \begin{center}
$a_{m-i}^0,c_{m-i}^0, a_1^1,\ldots, a_{m-i-1}^1, d_1^1,\ldots, d_{i-1}^1, b_{i-1}^1,$\\
$ a_1^2,\ldots, a_{m-i-1}^2, d_1^2,\ldots, d_{i-2}^2, b_{i-2}^2,c_{m-i-1}^2,  a_1^3,\ldots, a_{m-i-2}^3, d_1^3,\ldots, d_{i-2}^3, b_{i-2}^3,c_{m-i-2}^3, s.$
\end{center}
 with the relation $(a_{m-i}^0-a^1_{m-i-1})b_{i-1}^1c_{m-i}^0=s$.
 Since $I_s^1<I_s^2$, by equations \eqref{eq3} and in particular \begin{center}
$c^1_{m-i-1}=(d_{i-1}^1-d_{i-2}^2)c_{m-i-1}^2,\;d_{i-2}^1=(d_{i-1}^1-d_{i-2}^2)d_{i-2}^2$
\end{center}

\begin{center}
$a_{m-i}^0,c_{m-i}^0, d_{i-1}^1, b_{i-1}^1,$\\
$a_1^2,\ldots, a_{m-i-1}^2, d_1^2,\ldots, d_{i-2}^2, b_{i-2}^2,c_{m-i-1}^2, a_1^3,\ldots, a_{m-i-2}^3, d_1^3,\ldots, d_{i-2}^3, b_{i-2}^3,c_{m-i-2}^3, s.$
\end{center}
 with the relation \begin{center}
  $(a_{m-i}^0-a_{m-i-1}^2-b_{i-1}^1c_{m-i-1}^2)b^1_{i-1}c_{m-i}^0=s$ and $c_{m-i}^0(a_{m-i}^0-a_{m-i-1}^2-b_{i-1}^1c_{m-i-1}^2)=(d_{i-1}^1-d_{i-2}^2)c_{m-i-1}^2$.
 \end{center}
Since $I_s^2<I_s^3$, by equations \eqref{eq3}  and in particular \begin{center}
$c^3_{m-i-2}=c^2_{m-i-1}(a^2_{m-i-1}-a^3_{m-i-2}),\; (a^2_{m-i-1}-a^3_{m-i-2})b^3_{i-2}c^2_{m-i-1}=s$\end{center}
gives

\begin{center}
$a_{m-i}^0,c_{m-i}^0, d_{i-1}^1, b_{i-1}^1, a_{m-i-1}^2, c_{m-i-1}^2, a_1^3,\ldots, a_{m-i-2}^3, d_1^3,\ldots, d_{i-2}^3, b_{i-2}^3, s$
\end{center} with the relation 
\begin{center}
$(a_{m-i}^0-a_{m-i-1}^2-b_{i-1}^1c_{m-i-1}^2)b^1_{i-1}c_{m-i}^0=s,$\\
$( b^2_{i-2}=\;\;)(d_{i-1}^1-(d_{i-2}^3+c_{m-i-1}^2b_{i-2}^3))b_{i-1}^1=(a_{m-i-1}^2-a_{m-i-2}^3)b^3_{i-2}\; and\;$\\
 $(c_{m-i-1}^1= \;\;)(d_{i-1}^1-(d_{i-2}^3+c_{m-i-1}^2b_{i-2}^3))c_{m-i-1}^2=(a_{m-i}^0-(a_{m-i-1}^2+b_{i-1}^1c_{m-i-1}^2))c_{m-i}^0.$ \end{center}

Hence we get a final one
\begin{center}
$a_{m-i}^0,c_{m-i}^0, d_{i-1}^1, b_{i-1}^1, a_{m-i-1}^2, c_{m-i-1}^2, a_1^3,\ldots, a_{m-i-2}^3, d_1^3,\ldots, d_{i-2}^3, b_{i-2}^3$\\
\end{center}
with two relations \begin{center} $( b^2_{i-2}=\;\;)(d_{i-1}^1-(d_{i-2}^3+c_{m-i-1}^2b_{i-2}^3))b_{i-1}^1=(a_{m-i-1}^2-a_{m-i-2}^3)b^3_{i-2}$ \end{center}
and \begin{center}
$(c_{m-i-1}^1= \;\;)(d_{i-1}^1-(d_{i-2}^3+c_{m-i-1}^2b_{i-2}^3))c_{m-i-1}^2=(a_{m-i}^0-(a_{m-i-1}^2+b_{i-1}^1c_{m-i-1}^2))c_{m-i}^0.$\end{center}
It is complete intersection since \begin{center}
$(d_{i-1}^1-(d_{i-2}^3+c_{m-i-1}^2b_{i-2}^3))b_{i-1}^1-(a_{m-i-1}^2-a_{m-i-2}^3)b^3_{i-2},\; (d_{i-1}^1-(d_{i-2}^3+c_{m-i-1}^2b_{i-2}^3))c_{m-i-1}^2-(a_{m-i}^0-(a_{m-i-1}^2+b_{i-1}^1c_{m-i-1}^2))c_{m-i}^0$ \end{center} 
form a regular sequence in the polynomial ring $\mathbb{C}[a_{m-i}^0,c_{m-i}^0,..
  d_1^3,\ldots, d_{i-2}^3, b_{i-2}^3 ]$

The flag \text{Hilb}ert scheme $\text{Hilb}_{m_{\bullet}}(R)$ at this point has a local equation\\
 $(a_{m-i}^0-(a_{m-i-1}^2+b_{i-1}^1c_{m-i-1}^2))b^1_{i-1}c_{m-i}^0=(d_{i-1}^1-(d_{i-2}^3+c_{m-i-1}^2b_{i-2}^3))b^1_{i-1}c_{m-i-1}^2$ $=(a^2_{m-i-1}-a^3_{m-i-2})b^3_{i-2}c^2_{m-i-1}=0$
Let $Q_i^m=(x^{m-i+1},y^{i}),\; Q_i^{m-1}=(x^{m-i},y^{i}), \;Q_i^{m-2}=(x^{m-i-1},y^{i}),\; Q_{i}^{m-3}=(x^{m-i-2},y^{i})$. Then  $(Q_i^m, Q_i^{m-1}, Q_i^{m-2}, Q_{i}^{m-3}) \in \text{Hilb}^0_{m_\bullet}(R)$. Consider  a pair $(\;I_s^0<I_s^1<I_s^2<I_s^3)$ flatly deforming $(Q_i^m, Q_i^{m-1}, Q_i^{m-2}, Q_i^{m-3})$ relative to B. Then we may assume that for some $s \in m_s$,  $I_s^k$ is generated by $xy-s$ and $f^k, g^k$ for $k=1,2,3,4.$

(4) At the point $(Q_i^m, Q_{i-1}^{m-1}, Q_{i-1}^{m-2}, Q_{i-1}^{m-3})\in \text{Hilb}_{m_\bullet}^0(R)$. 
Let $(I_s^0<I_s^1<I_s^2<I_s^3)$ be flatly deforming $Q_i^m<
Q_{i-1}^{m-1}<Q_{i-1}^{m-2}<Q_{i-1}^{m-3}$. Then we may assume that for some $s \in m_s$,  $I_s^k$ is generated by $xy-s$ and $f^k, g^k$ for $k=1,2,3,4.$  where \begin{align}
\notag&f^0=x^{m-i+1}+\sum^{m-i}_{j=0}a_j^0x^j+\sum^{i-1}_{j=1}b_j^0y^j,\;g^0=y^{i}+\sum^{m-i}_{j=0}c_j^0x^j+\sum^{i-1}_{j=1}d_j^0y^j,\\
\notag&f^1=x^{m-i+1}+\sum^{m-i}_{j=0}a_j^1x^j+\sum^{i-2}_{j=1}b_j^1y^j,\;g^1=y^{i-1}+\sum^{m-i}_{j=0}c_j^1x^j+\sum^{i-2}_{j=1}d_j^1y^j,\\
\notag&f^2=x^{m-i}+\sum^{m-i-1}_{j=0}a_j^2x^j+\sum^{i-2}_{j=1}b_j^2y^j,\;g^2=y^{i-1}+\sum^{m-i-1}_{j=0}c_j^2x^j+\sum^{i-2}_{j=1}d_j^2y^j,\\
\notag&f^3=x^{m-i-1}+\sum^{m-i-2}_{j=0}a_j^3x^j+\sum^{i-2}_{j=1}b_j^3y^j,\;g^3=y^{i-1}+\sum^{m-i-2}_{j=0}c_j^3x^j+\sum^{i-2}_{j=1}d_j^3y^j.
\end{align}

 \eqref{eq2} tells that all possible parameters for the relative four flag \text{Hilb}ert scheme at  this point are
\begin{center}
$a_1^0,\ldots, a_{m-i}^0, d_1^0,\ldots,d_{i-1}^0, b_{i-1}^0,c_{m-i}^0,a_1^1,\ldots, a_{m-i}^1, d_1^1,\ldots, d_{i-2}^1, b_{i-2}^1,c_{m-i}^1,$\\
$ a_1^2,\ldots,a_{m-i-1}^2, d_1^2,\ldots, d_{i-2}^2, b_{i-2}^2,c_{m-i-1}^2, a_1^3,\ldots, a_{m-i-2}^3, d_1^3,\ldots, d_{i-2}^3, b_{i-2}^3,c_{m-i-2}^3, s.$
\end{center}
Since $I_s^0<I_s^1$, 
by equations \eqref{eq3} and in particular \begin{center}
$c_{m-i}^0=(d_{i-1}^0-d_{i-2}^1)c_{m-i}^1,\;b_{i-2}^1=(d_{i-1}^0-d_{i-2}^1)b_{i-1}^0$
\end{center}
allow us to eliminate parameters, above will be reduced to
\begin{center}
$d_{i-1}^0, b_{i-1}^0,  a_1^1,\ldots, a_{m-i}^1, d_1^1,\ldots, d_{i-2}^1, b_{i-2}^1,c_{m-i}^1,$\\
$ a_1^2,\ldots, a_{m-i-1}^2, d_1^2,\ldots,d_{i-2}^2, b_{i-2}^2,c_{m-i-1}^2, a_1^3,\ldots, a_{m-i-2}^3, d_1^3,\ldots, d_{i-2}^3, b_{i-2}'^3,c_{m-i-2}^3,s$
\end{center}
Since $I_s^1<I_s^2$, we have  $c^2_{m-i-1}=c^1_{m-i}(a^1_{m-i}-a^2_{m-i-1})$,  $a_j^1=a_{j-1}^2-(a_{m-i}^1-a_{m-i-1}^2)a_j^2$ \;and\; $d_j^1=d_j^2+c_{m-i}^1b_j^2$ for $j=1,..i-2$ and in particular
\begin{center}
$c^2_{m-i-1}=c^1_{m-i}(a^1_{m-i}-a^2_{m-i-1}),\; b_{i-2}^1=(a_{m-i}^1-a_{m-i-1}^2)b_{i-2}^2$
\end{center} allow us to eliminate to less parameters
\begin{center}
$d_{i-1}^0, b_{i-1}^0, a_{m-i}^1, c_{m-i}^1,$\\ $a_1^2,\ldots, a_{m-i-1}^2, d_1^2,\ldots, d_{i-2}^2, b_{i-2}^2,$
$ a_1^3,\ldots, a_{m-i-2}^3, d_1^3,\ldots, d_{i-2}^3, b_{i-2}^3,c_{m-i-2}^3 ,s. $
\end{center}
Since $I_s^2<I_s^3$, we have  $c^3_{m-i-2}=c^2_{m-i-1}(a^2_{m-i-1}-a^3_{m-i-2})$, $a_j^2=a_{j-1}^3-(a_{m-i-1}^2-a_{m-i-2}^3)a_j^3, j=1,..,m-i-2$ and $d_j^2=d_j^3+c_{m-i}^2b_j^3, j=1,..i-2$(OR by \eqref{eq2}) and in particular
$$c^3_{m-i-2}=c^2_{m-i-1}(a^2_{m-i-1}-a^3_{m-i-2}),\;b_{i-2}^2=(a_{m-i-1}^2-a_{m-i-2}^3)b_{i-2}^3.$$

$$d^0_{i-1}, b^0_{i-1}, a^1_{m-i}, c^1_{m-i},  a^2_{m-i-1}, a^3_1,\ldots, a^3_{m-i-2}, d^3_1,\ldots, d^3_{i-2}, b^3_{i-2}$$
with one relation$(b_{i-2}^1=)$ $$(a^1_{m-i}-a^2_{m-i-1})(a^2_{m-i-1}-a^3_{m-i-2})b^3_{i-2}=\{d^0_{i-1}-d_{i-2}^3-(a^1_{m-i}-a^3_{m-i-2})c_{m-i}^1b_{i-2}^3\}b_{i-1}^0.$$
Thus the relative flag \text{Hilb}ert scheme $\text{Hilb}_{m_{\bullet}}(\tilde{R}/B)$ at the points $(Q_i^m, Q_i^{m-1}, Q_i^{m-2}, Q_{i-1}^{m-3})$ has a hypersurface singularity 
$$(a^1_{m-i}-a^2_{m-i-1})(a^2_{m-i-1}-a^3_{m-i-2})b^3_{i-2}=\{d^0_{i-1}-d_{i-2}^3-(a^1_{m-i}-a^3_{m-i-2})c_{m-i}^1b_{i-2}^3\}b_{i-1}^0.$$
For the flag \text{Hilb}ert scheme $\text{Hilb}_{m_{bullet}}(R)$ at this points
$$(a^1_{m-i}-a^2_{m-i-1})(a^2_{m-i-1}-a^3_{m-i-2})b^3_{i-2}c_{m-i}^1$$
$$=\{d^0_{i-1}-d_{i-2}^3-(a^1_{m-i}-a^3_{m-i-2})c_{m-i}^1b_{i-2}^3\}b_{i-1}^0c_{m-i}^1=0.$$
$$d^0_{i-1}, b^0_{i-1}, a^1_{m-i}, c^1_{m-i}, a^2_{m-i-1}, a^3_1,\ldots, a^3_{m-i-2}, d^3_1,\ldots, d^3_{i-2}, b^3_{i-2}$$
with the relation
$$(a^1_{m-i}-a^2_{m-i-1})(a^2_{m-i-1}-a^3_{m-i-2})b^3_{i-2}=\{d^0_{i-1}-d_{i-2}^3-(a^1_{m-i}-a^3_{m-i-2})c_{m-i}^1b_{i-2}^3\}b_{i-1}^0$$
Thus the relative flag \text{Hilb}ert scheme $\text{Hilb}_{m_{\bullet}}(\tilde{R}/B)$ at the points $(Q_i^m, Q_i^{m-1}, Q_i^{m-2}, Q_{i-1}^{m-3})$ has a hypersurface singularity 
$$(a^1_{m-i}-a^2_{m-i-1})(a^2_{m-i-1}-a^3_{m-i-2})b^3_{i-2}=\{d^0_{i-1}-d_{i-2}^3-(a^1_{m-i}-a^3_{m-i-2})c_{m-i}^1b_{i-2}^3\}b_{i-1}^0.$$For the flag \text{Hilb}ert scheme at this point, it has a local euqation
$$(a^1_{m-i}-a^2_{m-i-1})(a^2_{m-i-1}-a^3_{m-i-2})b^3_{i-2}c_{m-i}^1$$ $$=\{d^0_{i-1}-d_{i-2}^3-(a^1_{m-i}-a^3_{m-i-2})c_{m-i}^1b_{i-2}^3\}b_{i-1}^0c_{m-i}^1=0.$$
\end{proof}


\subsection{In General}

Set $m_{\bullet}=(m,m-1,\ldots,m-n+1)$
Consider the sequence of ideals $\{Q^k_j\}_{j,k}$ where the upper index $\{k\}$ is decreasing by one and the lower index $\{j\}$ is decreasing by one or equal.
Let $A_h$ be the sequence of ideals of $\{Q^k_j\}$ with  the lower indices equal and $B_h$ with the lower indices decreasing by one. We will call these as the block A and the block B and h gives  the position of the block in the sequence. Let $|A_h|$ and $|B_h|$ be the length of sequence of ideals $\{Q^k_j\}$.
We will choose the block A as long as possible such that the length of  the block A is at least 2.

\begin{lemma} $\text{Hilb}_{m_{\bullet}}(\tilde{R}/B)$ at the point $A_1=(Q_i^m,Q_i^{m-1},\ldots,, Q_i^{m-n},Q_i^{m-n+1})$\;with\;$|A_1|=n$\; is m+1 dimensional smooth scheme.\\
\begin{proof}

$$a_{m-i}^0,c_{m-i}^0, a_{m-i-1}^{1}, a_{m-i-2}^{2},a_{m-i-3}^{3},\ldots, a_{m-i-n+2}^{n-2}, a_{1}^{n-1},\ldots, a_{m-i-n+1}^{n-1},$$ $$d_1^{n-1},d_2^{n-1},\ldots,d^{n-1}_{i-1},b^{n-1}_{i-1}$$
are regular parameters.
\end{proof}
\end{lemma}
Let  $A_1A_2$ be  $(Q_i^m,Q_i^{m-1},\ldots Q_i^{m-k+1},Q_{i-1}^{m-k},\ldots,Q_{i-1}^{m-n+1}),\; k\geqslant 2$\; with\; $|A_1|=k$\;and\; $|A_1|+|A_2|=n$
\begin{lemma}
$\text{Hilb}_{m_{\bullet}}(\tilde{R}/B)$ at the point $A_1A_2$  is locally complete intersection embedded in $\mathbb{A}^{m+4}$ with two equations [AAC] (Between block A and block A, we have one equation coming from c)
\begin{align}
&(d_{i-1}^{k-1}-d^k_{i-2})b^{k-1}_{i-1}\notag\\&=\prod^{n-2}_{j=k}(a_{m-i-j+1}^j-a_{m-i-j}^{j+1})b^{n-1}_{i-2}
\notag
\end{align}
 and [AAB] (Between block A and block A coming from b) 
\begin{align}
&(d_{i-1}^{k-1}-d_{i-2}^{k})c_{m-i-k+1}^k\notag\\&=\prod^{k-3}_{j=0}(a_{m-i-j}^j-a_{m-i-j-1}^{j+1})\left(a_{m-i-k+2}^{k-2}-a^{k-1}_{m-i-k+1}\right)c_{m-i}^0.\notag 
\end{align}
where $d_{i-2}^k$ is as below and $a^{k-1}_{m-i-k+1}=a_{m-i-k+1}^{k}+b_{i-1}^{k-1}c_{m-i-k+1}^k$.
 \end{lemma}
\begin{proof} There are regular parameters(see the table below)

 $$a^0_{m-i},c^0_{m-i}, a^1_{m-i-1},\ldots, a^{k-2}_{m-i-k+2}, d^{k-1}_{i-1},b^{k-1}_{i-1},  a^k_{m-i-k+1},c^k_{m-i-k+1}, $$ $$a^{k+1}_{m-i-k},a^{k+2}_{m-i-k-1},\ldots,a^{n-2}_{m-n-i+3}, a_1^{n-1},\ldots, a_{m-n-i+2}^{n-1},d_1^{n-1},\ldots,d_{i-2}^{n-1},b_{i-2}^{n-1}$$
with the equations
$$(b_{i-2}^k=)( d_{i-1}^{k-1}-d_{i-2}^k)b^{k-1}_{i-1}=\prod^{n-2}_{j=k}(a_{m-i-j+1}^j-a_{m-i-j}^{j+1})b^{n-1}_{i-2}$$ %
and
$$(c^{k-1}_{m-k-i+1}=) ( d_{i-1}^{k-1}-d_{i-2}^k)b_{i-2}^{n-1}c_{m-i-k+1}^k$$
$$=\prod^{k-3}_{j=0}(a_{m-i-j}^j-a_{m-i-j-1}^{j+1})(a_{m-i-k+2}^{k-2}-(a_{m-i-k+1}^{k}+b_{i-1}^{k-1}c_{m-i-k+1}^k))c_{m-i}^0.$$
and in fact it is local complete intersection where $d^k_{i-2}$  is as below.
Using \eqref{eq3} repeatedly, we can get
 $$\mathbf{d^{k}_{i-2}=}$$  $$=d^{k+d}_{i-2}+\sum_{l=k}^{d+k-1} \prod_{j\geqslant k, j\neq l}^{d+k-1}(a^{j}_{m-i-j+1}-a^{j+1}_{m-i-j}) c^k_{{m-i-k+1}}b^{k+d}_{i-2}.$$
 Thus if $d=n-k -1$, then  $$d^k_{i-2}=d^{n-1}_{i-2}+\sum_{l=k}^{n-2} \prod_{j\geqslant k, j\neq l}^{n-2}(a^{j}_{m-i-j+1}-a^{j+1}_{m-i-j}) c^k_{{m-i-k+1}}b^{n-1}_{i-2}$$
\begin{center}
 $\left.\begin{array}{ccccc|ccccc}
 &&A_1&&&&A_2&&&\\\hline
Q_i^m&Q_i^{m-1}&\ldots&Q_i^{m-k+2}&Q_i^{m-k+1}&Q_{i-1}^{m-k}&Q_{i-1}^{m-k-1}&\ldots&Q_{i-1}^{m-n+2}&Q_{i-1}^{m-n+1}\\\hline
  a^0_{m-i}&a^1_{m-i-1} &\ldots&a^{k-2}_{m-i-k+2} & b^{k-1}_{i-1}& a^k_{m-i-k+1} & a^{k+1}_{m-i-k}&\ldots &a^{n-2}_{m-n-i+3} &a_1^{n-1} \\c^0_{m-i} & & && d^{k-1}_{i-1}&&&&&\ldots \\&&&&&& &  & &a_{m-n-i+2}^{n-1}\\&&&&&&& & &  d_1^{n-1}\\&  &  &  &  &  &  & && \ldots \\&  &  &   &  &  &  & & &d_{i-2}^{n-1}\\
  &&&&(c^{k-1})&(b^{k})&&&&b_{i-2}^{n-1}\\\hline\end{array}\right.$
  \end{center}
\end{proof} 
Let $A_1B_2$ be $(Q_i^m, Q_i^{m-1},\ldots,Q_i^{m-k+1},Q_{i-1}^{m-k},\ldots,Q_{i-n+k}^{m-n+1})$ with $|A_1|=k and |A_1|+|B_2|=n$ 
\begin{lemma}
$\text{Hilb}_{m_{\bullet}}(\tilde{R}/B)$ at the points $A_1B_2$ has a hypersurface singularity in $\mathbb{A}^{m+1}$ with a equation
\begin{align}
&\prod^{n-k-1}_{j=0}(d^{k+j-1}_{i-j-1}-d^{k+j}_{i-j-2})c^{n-1}_{m-i-k+1}\notag\\&=\prod^{k-2}_{j=0}(a^{j}_{m-i-j}-a^{j+1}_{m-i-j-1})c^0_{m-i}.\notag 
\end{align}
and it is a local complete intersection where $a^{k-1}_{m-i-k+1}$ is as below.
\end{lemma}
\begin{proof}There are regular parameters 
\begin{center}
$a^{0}_{m-i},c^{0}_{m-i}, a^{1}_{m-i-1},a^2_{m-i-2}\ldots, a^{k-2}_{m-i-k+2}, d^{k-1}_{i-1},b^{k-1}_{i-1},$\\
$d^k_{i-2},d^{k+1}_{i-3},\ldots,d^{n-2}_{i-n+k}, a^{n-1}_1, a^{n-1}_2,\ldots, a^{n-1}_{m-i-k+1}, d^{n-1}_1,d^{n-1}_2,\ldots,d^{n-1}_{i-n+k-1},c^{n-1}_{m-i-k+1}$\\
\end{center}
with the equation [ABC]
 $$(c^{k-1}_{m-i-k+1}=)\prod^{n-1}_{j=k}(d^{j-1}_{i-j+k-1}-d^{j}_{i-j+k-2})c^{n-1}_{m-i-k+1}=\prod^{k-2}_{j=0}(a^{j}_{m-i-j}-a^{j+1}_{m-i-j-1})c^0_{m-i}$$
 where $a^{k-1}_{m-i-k+1}$ is as below. 
Use \eqref{eq4} repeatedly,
 $$\mathbf{a^{k-1}_{m-i-k+1}=}$$ 
$$=a^{k+d}_{m-i-k+1}+ \sum_{l=k}^{k+d} \prod_{j\geqslant k, j\neq l}^{k+d}(d^{j-1}_{i-j+k-1}-d^{j}_{i-j+k-2})b^{k-1}_{i-1}c^{k+d}_{m-i-k+1}.$$
If $d=n-k-1$, then
$$a^{k-1}_{m-i-k+1}=a^{n-1}_{m-i-k+1}+ \sum_{l=k}^{n-1} \prod_{j\geqslant k, j\neq l}^{n-1}(d^{j-1}_{i-j+k-1}-d^{j}_{i-j+k-2})b^{k-1}_{i-1}c^{n-1}_{m-i-k+1}$$
\begin{center}
 
$\left.\begin{array}{ccccc|ccccc}
 &&A_1&&&&B_2&&&\\\hline
Q_i^m&Q_i^{m-1}&\ldots&Q_i^{m-k+2}&Q_i^{m-k+1}&Q_{i-1}^{m-k}&Q_{i-2}^{m-k-1}&\ldots&Q_{i-1}^{m-n+2}&Q_{i-n+k}^{m-n+1}\\\hline
  a^0_{m-i}&a^1_{m-i-1} &\ldots&a^{k-2}_{m-i-k+2} & b^{k-1}_{i-1}& d^k_{i-2} & d^{k+1}_{i-3}&\ldots &d^{n-2}_{i-n+k} &a_1^{n-1} \\c^0_{m-i} & & && d^{k-1}_{i-1}&&&&&\ldots \\&&&&&& &  & &a_{m-i-k+1}^{n-1}\\&&&&&&& & &  d_1^{n-1}\\&  &  &  &  &  &  & && \ldots \\&  &  &   &  &  &  & & &d_{i-n+k-1}^{n-1}\\&&&&(c^{k-1})&&&&&b_{i-n+k-1}^{n-1}\\\hline\end{array}\right.$
 \end{center}
   \end{proof}
Let  $B_1A_2$ be $(Q^m_i,Q^{m-1}_{i-1},Q^{m-2}_{i-2},\ldots,Q^{m-k}_{i-k},Q^{m-k-1}_{i-k},\ldots,Q_{i-k}^{m-n-1} )$ with $|B_1|=k$ and $|B_1|+|A_2|=n.$
\begin{lemma}
$\text{Hilb}_{m_{\bullet}}(\tilde{R}/B)$ at the points $B_1A_2$ has a hypersurface singularity.
\end{lemma}
\begin{proof}
There are regular parameters
$$b^0_{i-1},d^0_{i-1}, d^1_{i-2},,,,d^{k-1}_{i-k},\\ a^{k}_{m-i},c^{k}_{m-i}, a^{k+1}_{m-i-1},\ldots, a^{n-2}_{m-i-n+k+2},$$ $$a^{n-1}_1,\ldots, a ^{n-1}_{m-i-n+k+1},d^{n-1}_{1},\ldots, d^{n-1}_{i-k-1},b^{n-1}_{i-k-1}$$
with only one equation [BAB]
$$(b^{k}_{i-k-1}=) \prod^{k-1}_{j=0} (d_{i-j}^{j}-d_{i-j-1}^{j+1})b_{i-1}^0$$
$$=\prod_{j=k_1}^{n-2} (a^{j}_{m-i-j+k}-a^{j+1}_{m-i-j+k-1})b^{n-1}_{i-k-1}$$
where $d^{k}_{i-k-1}$ is as below(same as one in Lemma 9.) 
$$d^{k}_{i-k-1}=d^{n-1}_{i-k-1}+\sum_{l=k}^{n-2} \prod_{j\geqslant k, j\neq l}^{n-2}(a^{j}_{m-i-j+1}-a^{j+1}_{m-i-j}) c^{k}_{{m-i}}b^{n-1}_{i-k-1}$$
\begin{center}
$\left.\begin{array}{ccccc|ccccc}
 &&B_1&&&&A_2&&&\\\hline
Q_i^m&Q_{i-1}^{m-1}&\ldots&Q_{i-k+2}^{m-k+2}&Q_{i-k+1}^{m-k+1}&Q_{i-k}^{m-k}&Q_{i-k}^{m-k-1}&\ldots&Q_{i-k}^{m-n+2}&Q_{i-k}^{m-n+1}\\\hline
  b^0_{i-1}&d^1_{i-2} &\ldots&d^{k-2}_{i-k+1} & d^{k-1}_{i-k}& a^k_{m-i} & a^{k+1}_{m-i-1}&\ldots &a^{n-2}_{m-i-n+k+2}&a_1^{n-1} \\d^0_{i-1} & & &&& c^{k-1}_{m-i}&&&&\ldots \\&&&&&& &  & &a_{m-n-i+k+1}^{n-1}\\&&&&&&& & &  d_1^{n-1}\\&  &  &  &  &  &  & && \ldots \\&  &  &   &  &  &  & & &d_{i-k-1}^{n-1}\\
  &&&&&(b^{k})&&&&b_{i-k-1}^{n-1}\\\hline\end{array}\right.$
 \end{center}
\end{proof}
\begin{lemma} 
The relative \text{Hilb}ert scheme at the points
$A_1A_2\ldots A_h$ where $k_j-k_{j-1}\geqslant 2.$  $|A_1|=k_1, |A_1|+|A_2|=k_2,\ldots,\sum^{h}_{s=1}|A_s|=k_h$ is locally complete intersection embedded in $\mathbb{A}^{2(h-1)+m+1}$ with $2(h-1)$ (independent) equations.
\end{lemma}
\begin{proof}
By lemma 2, let us see up to $A_1A_2$.

$$a_{m-i}^0,c_{m-i}^0,$$
$$a_{m-i-1}^1,a_{m-i-2}^{2},\ldots,a_{m-i-k_1+2}^{k_1-2},$$
$$d_{i-1}^{k_1-1},b_{i-1}^{k_1-1}$$
$$a_{m-i-k_1+1}^{k_1},c_{m-i-k_1+1}^{k_1}$$
$$a_{m-i-k_1}^{k_1+1},\ldots, a_{m-i-k_2+3}^{k_2-2}$$
$$a_1^{k_2-1},\ldots, a_{m-k_2-i+2}^{k_2-1},d_1^{k_2-1},\ldots,d_{i-2}^{k_2-1},b_{i-2}^{k_2-1}$$

 with the equations
$(b_{i-2}^ {k_1}=)$ $$\{ d_{i-1}^{ {k_1}-1}-(d_{i-2}^{ {k_2}-1}+\sum_{l=k_1}^{k_2-2}\prod^{ {k_2}-2}_{j\geqslant {k_1}, j\neq l}(a_{m-i-j+1}^j-a_{m-i-j}^{j+1})b_{i-2}^{{k_2}-1}c_{m-i- {k_1}+1}^ {k_1})\}b^{ {k_1}-1}_{i-1}$$ $$=\prod^{ {k_2}-2}_{j= {k_1}}(a_{m-i-j+1}^j-a_{m-i-j}^{j+1})b^{ {k_2}-1}_{i-2}$$ 
and 
$(c^{ {k_1}-1}_{m- {k_1}-i+1}=)$ $$ \{d_{i-1}^{k_1-1}-(d_{i-2}^{ {k_2}-1}+\sum^{k_2-2}_{l=k_1}\prod^{ {k_2}-2}_{j\geqslant{k_1}, j\neq l}(a_{m-i-j+1}^j-a_{m-i-j}^{j+1})b_{i-2}^{ {k_2}-1}c_{m-i- {k_1}+1}^ {k_1})\}c_{m-i- {k_1}+1}^ {k_1}$$
$$=\prod^{ {k_1}-3}_{j=0}(a_{m-i-j}^j-a_{m-i-j-1}^{j+1})(a_{m-i- {k_1}+2}^{ {k_1}-2}-(a_{m-i- {k_1}+1}^{ {k_1}}+b_{i-1}^{ {k_1}-1}c_{m-i- {k_1}+1}^ {k_1}))c_{m-i}^{k_1}.$$

Up to $A_1A_2A_3$, by lemma 2 for $A_2A_3$, we can get
$$a_{m-i}^0,c_{m-i}^0,$$
$$a_{m-i-1}^1,a_{m-i-2}^{2},\ldots,a_{m-i-k_1+2}^{k_1-2},$$
$$d_{i-1}^{k_1-1},b_{i-1}^{k_1-1}$$
$$a_{m-i-k_1+1}^{k_1},c_{m-i-k_1+1}^{k_1}$$
$$a_{m-i-k_1}^{k_1+1},\ldots, a_{m-i-k_2+3}^{k_2-2}$$
$$d_{i-2}^{k_2-1},b_{i-2}^{k_2-1}$$
$$a_{m-i-k_2+2}^{k_2}, c_{m-i-k_2+2}^{k_2}$$
$$a_{m-i-k_2+1}^{k_2+1},\ldots, a_{m-i-k_3+4}^{k_3-2},$$
$$a_1^{k_3-1},\ldots, a_{m-k_3-i+3}^{k_3-1},d_1^{k_3-1},\ldots,d_{i-2}^{k_3-1},b_{i-2}^{k_3-1}$$

with two more equations
$$(b_{i-2}^ {k_2}=)(d_{i-1}^{ {k_2}-1}-d^{k_2}_{i-2})b^{ {k_2}-1}_{i-1}=\prod^{ {k_3}-2}_{j= {k_2}}(a_{m-i-j+1}^j-a_{m-i-j}^{j+1})b^{ {k_3}-1}_{i-2}$$
and 
$$(c^{ {k_2}-1}_{m- {k_2}-i+1}=)(d_{i-1}^{ {k_2}-1}-d^{k_2}_{i-2})c_{m-i- {k_2}+1}^ {k_2}$$
$$=\prod^{ {k_2}-3}_{j=k_1}(a_{m-i-j}^j-a_{m-i-j-1}^{j+1})(a_{m-i- {k_2}+2}^{ {k_2}-2}-a^{k_2-1}_{m-i- {k_2}+1})c_{m-i-k_1+1}^{k_1}.$$
 where $d^{k_2}_{i-2}$ and $a^{k_2-1}_{m-i- {k_2}+1}$ as below. 
 $$d^{k_2}_{i-2}=d_{i-2}^{ {k_3}-1}+\sum^{k_3-2}_{l=k_2}\prod^{ {k_3}-2}_{j\geqslant {k_2}, j\neq l}(a_{m-i-j+1}^j-a_{m-i-j}^{j+1})b_{i-2}^{{k_2}-1}c_{m-i-k_2+2}^{k_2}$$
$$a^{k_2-1}_{m-i- {k_2}+1}=a_{m-i- {k_2}+1}^{ {k_2}}+b_{i-1}^{ {k_2}-1}c_{m-i- {k_2}+1}^ {k_2}$$
Similarly, for $A_1A_2\ldots A_h $ we have parameters as below

$$a_{m-i}^0,c_{m-i}^0,$$
$$a_{m-i-1}^1,a_{m-i-2}^{2},\ldots,a_{m-i-k_1+2}^{k_1-2},$$
$$d_{i-1}^{k_1-1},b_{i-1}^{k_1-1}$$
$$a_{m-i-k_1+1}^{k_1},c_{m-i-k_1+1}^{k_1}$$
$$a_{m-i-k_1}^{k_1+1},\ldots, a_{m-i-k_2+3}^{k_2-2}$$
$$d_{i-2}^{k_2-1},b_{i-2}^{k_2-1}$$
$$a_{m-i-k_2+2}^{k_2}, c_{m-i-k_2+2}^{k_2}$$
$$a_{m-i-k_2+1}^{k_2+1},\ldots, a_{m-i-k_3+4}^{k_3-2},$$
$$d_{i-3}^{k_3-1}, b_{i-3}^{k_3-1}$$
$$a_{m-i-k_3+3}^{k_3},c_{m-i-k_3+3}^{k_3}$$
$$a_{m-i-k_3+2}^{k_3+1},\ldots, a_{m-i-k_4+5}^{k_4-2}$$
$$\ldots$$
$$\ldots$$
$$\ldots$$
$$\ldots$$
$$d_{i-h+1}^{k_{h-1}-1},b_{i-h+1}^{k_{h-1}-1}$$
$$a_{m-i-k_{h-1}+h-1}^{k_{h-1}},c_{m-i-k_{h-1}+h-1}^{k_{h-1}}$$
$$a_{m-i-k_{h-1}+h-2}^{k_{h-1}+1},\ldots, a_{m-i-k_h+h+1}^{k_{h}-2}$$
$$a_1^{k_{h}-1},\ldots, a_{m-i-k_h+h}^{k_{h}-1}, d_1^{k_{h}-1}\ldots,d_{i-h}^{k_{h}-1}, b_{i-h}^{k_{h}-1}$$
with the set of equations [[AAB]]  
$$(b_{i-j-1}^{k_j}=)$$
\begin{align}
&(d_{i-j}^{k_j-1}-d_{i-j-1}^{k_j})b_{i-j}^{k_j-1}\notag\\
&=\prod^{k_{j+1}-3}_{e=k_j}(a^{e}_{m-i+j-e}-a^{e+1}_{m-i+j-e-1})(a^{k_{j+1}-2}_{m-i-k_{j+1}+j+2}-a^{k_{j+1}-1}_{m-i-k_{j+1}+j+1})b_{i-j-1}^{k_{j+1}-1}\notag\\\;
\end{align}
$ for\; j=1,..,h-1.$

and the set of equations [[AAC]]
$$(c^{k_{j}-1}_{m-i-k_{j}+j}=)$$
 \begin{align}
&\prod_{e=k_{j-1}}^{k_{j}-3}(a^{e}_{m-i+j-e-1}-a^{e+1}_{m-i+j-e-2})(a^{k_{j}-2}_{m-i-k_{j}+j+1}-a^{k_{j}-1}_{m-i-k_{j}+j})c_{m-i-k_{j-1}+j-1}^{k_{j-1}}\notag\\&=(d_{i-j}^{k_j-1}-d_{i-j-1}^{k_j})c_{m-i-k_{j}+{j}}^{k_{j}}\notag\; 
\end{align}
$for\; j=1,\ldots,h-1.$ 
where \small \; \;$a^{k_{j}-1}_{m-i-k_j+j}$\; and $d^{k_j}_{i-j-1}$\; as\; below for\; j=1,\ldots,, h-1.
 $$a^{k_{j}-1}_{m-i-k_j+j}=a^{k_j}_{m-i-k_j+j}+b^{k_j-1}_{i-j}c^{k_j}_{m-i-k_j+j}$$
\; and 
\small$$ d^{k_j}_{i-j-1}= 
d^{k_{j+1}-1}_{i-j}+$$ $$\sum^{k_{j+1}-3}_{l=k_j}\prod^{k_{j+1}-3}_{e\geqslant k_j, e\neq l}(a^{e}_{m-i+j-e}-a^{e+1}_{m-i+j-e-1})(a^{k_{j+1}-2}_{m-i-k_{j+1}+j+2}-a^{k_{j+1}-1}_{m-i-k_{j+1}+j+1}))b^{k_{j+1}-1}_{i-j-1}c^{k_j}_{m-i-k_j+j}.$$ \end{proof}
  Let  $P$ be $A_1..A_{l_1-1}B_{l_1} A_{l_1+1}\ldots A_{l_2-1}B_{l_2}A_{l_2+1}\ldots A_{l_t-1}B_{l_t}A_{l_t+1}..A_{h}$ where
$|A_1|=k_1, |A_j|=k_j-k_{j-1}, \; for j=1,..,\widehat{l_1},\ldots,\widehat{l_t},k_0=0 \; and \;|B_j|=k_j-k_{j-1}$ for $l_1,\ldots,l_t.$
\begin{theorem}
The relative \text{Hilb}ert scheme at the point P is locally complete intersection embedded in $\mathbb{A}^{2h-2+m+1}$ with $2h-2$ equations.

\end{theorem}
\begin{proof}
\begin{step} Let us look at $A_1..A_{l_1-1}$
$$ a_{m-i}^0,c_{m-i}^0,$$
$$a_{m-i-1}^1,a_{m-i-2}^{2},\ldots,a_{m-i-k_1+2}^{k_1-2},$$
$$d_{i-1}^{k_1-1},b_{i-1}^{k_1-1}$$
$$a_{m-i-k_1+1}^{k_1},c_{m-i-k_1+1}^{k_1}$$
$$a_{m-i-k_1}^{k_1+1},\ldots, a_{m-i-k_2+3}^{k_2-2}$$
$$d_{i-2}^{k_2-1},b_{i-2}^{k_2-1}$$
$$a_{m-i-k_2+2}^{k_2}, c_{m-i-k_2+2}^{k_2}$$
$$a_{m-i-k_2+1}^{k_2+1},\ldots, a_{m-i-k_3+4}^{k_3-2},$$
$$d_{i-3}^{k_3-1}, b_{i-3}^{k_3-1}$$
$$\ldots$$
$$\ldots$$
$$d_{i-l_1+2}^{k_{ l_1-2}-1},b_{i-l_1+2}^{k_ { l_1-2}-1}$$
$$a_{m-i-k_{ l_1-2}+ l_1-2}^{k_{l_1-2}},c_{m-i-k_{ l_1-2}+ l_1-2}^{k_{ l_1-2}}$$
$$a_{m-i-k_{ l_1-2}+l_1-3}^{k_{ l_1-2}+1},\ldots, a_{m-i-k_{ l_1-1}+l_1}^{k_{ l_1-1}-2}$$
$$a_1^{k_{l_1-1}-1},\ldots, a_{m-i-k_{l_1-1}+l_1-1}^{k_{l_1-1}-1}, d_1^{k_{l_1-1}-1}\ldots,d_{i-l_1+1}^{k_{l_1-1}-1}, b_{i-l_1+1}^{k_{l_1-1}-1}$$
with the relations [[AAB]]
\begin{align}(b^{k_j}_{i-j-1}=)
&(d_{i-j}^{k_j-1}-d_{i-j-1}^{k_j})b_{i-j}^{k_j-1}\notag\\&=\prod^{k_{j+1}-3}_{e=k_j}(a^{e}_{m-i+j-e}-a^{e+1}_{m-i+j-e-1})(a^{k_{j+1}-2}_{m-i-k_{j+1}+j+2}-a^{k_{j+1}-1}_{m-i-k_{j+1}+j+1})b_{i-j-1}^{k_{j+1}-1}.\notag 
\end{align}


and the relations[[AAC]]
\begin{align}(c^{k_j-1}_{m-i-k_j+j}=)
&c_{m-i-k_{j-1}+j-1}^{k_{j-1}}\prod_{e=k_{j-1}}^{k_j-3}(a^{e}_{m-i+j-e-1}-a^{e+1}_{m-i+j-e-2})(a^{k_j-2}_{m-i-k_j+j+1}-a^{k_j-1}_{m-i-k_j+j})\notag\\&=(d_{i-j}^{k_j-1}-d_{i-j-1}^{k_j})c_{m-i-k_{j}+{j}}^{k_{j}}\notag 
\end{align}
$for\; j\;=1,2,\ldots, l_1-2\;$

\end{step}

\begin{step} Let us look at up to $A_1A_2\ldots A_{l_1-1}B_{l_1}$ 
Let us recall $A_1A_2\ldots A_{l_i-1}$ where $A_j=Q_{i-j+1},Q_{i-j+1}, \ldots, Q_{i-j+1}$ and $|A_j|=k_{j}-k_{j-1}\geqslant 2$.
$$ a_{m-i}^0,c_{m-i}^0,$$
$$a_{m-i-1}^1,a_{m-i-2}^{2},\ldots,a_{m-i-k_1+2}^{k_1-2},$$
$$d_{i-1}^{k_1-1},b_{i-1}^{k_1-1}$$
$$a_{m-i-k_1+1}^{k_1},c_{m-i-k_1+1}^{k_1}$$
$$a_{m-i-k_1}^{k_1+1},\ldots, a_{m-i-k_2+3}^{k_2-2}$$
$$d_{i-2}^{k_2-1},b_{i-2}^{k_2-1}$$
$$a_{m-i-k_2+2}^{k_2}, c_{m-i-k_2+2}^{k_2}$$
$$a_{m-i-k_2+1}^{k_2+1},\ldots, a_{m-i-k_3+4}^{k_3-2},$$
$$d_{i-3}^{k_3-1}, b_{i-3}^{k_3-1}$$
$$\ldots$$
$$\ldots$$
$$d_{i-l_1+2}^{k_{ l_1-2}-1},b_{i-l_1+2}^{k_ { l_1-2}-1}$$
$$a_{m-i-k_{ l_1-2}+ l_1-2}^{k_{ l_1-2}},c_{m-i-k_{ l_1-2}+ l_1-2}^{k_{ l_1-2}}$$
$$a_{m-i-k_{ l_1-2}+l_1-3}^{k_{ l_1-2}+1},\ldots, a_{m-i-k_{ l_1-1}+l_1}^{k_{l_1-1}-2}$$
$$d_{i-l_1+1}^{k_{l_1-1}-1}, b_{i-l_1+1}^{k_{l_1-1}-1}$$
$$d_{i-l_1}^{k_{l_1-1}}, d_{i-l_1-1}^{k_{l_1-1}},\ldots, d_{i-l_1+k_{l_1-1}-k_{l_1}+1}^{k_{l_1}-2}$$
$$a_1^{k_{l_1}-1},\ldots, a_{m-i-k_{l_1}+l_1}^{k_{l_1}-1}, d_1^{k_{l_1}-1},\ldots,d_{i-l_1+k_{l_1-1}-k_{l_1}}^{k_{l_1}-1}, b_{i-l_1+k_{l_1}-k_{l_1}}^{k_{l_1}-1}, c^{{l_1}-1}_{m-i-k_{l_1}+l_1}$$

with one equation [ABC] (Between the block $A_{l_1-1}$ and the block $B_{l_1}$, equation coming from c)
\begin{align}
\notag(c^{k_{l_1-1}-1}_{m-i-k_{l_1-1}+{l_1-1}}=)
\end{align}
\begin{align}
\notag&c^{k_{l_1}-1}_{m-i-k_{l_1-1}+l_1-1}\prod^{k_{l_1}-1}_{e=k_{l_1-1}}(d_{i-l_1+1-e+k_{l_1-1}}^{e-1}-d_{i-l_1-e-1+k_{l_1-1}}^{e})\notag\\=&\prod^{k_{l_1-1}-3}_{e=k_{l_1-2}}( a_{m-i+l_1-2-e}^{e}-a_{m-i+l_1-2-e-1}^{k_{l_1-2}+e+1})(a_{m-i-k_{ l_1-1}+l_1}^{k_{ l_1-1}-2}-a_{m-i-k_{l_1-1}+l_1-1}^{k_{l_1-1}-1})c_{m-i-k_{l_1-2}+l_1-2}^{k_{l_1-2}}\notag.
\end{align}

Thus totally 3 kinds from Step 1 and above,

with  relations [[AAB]]
\begin{align}
\notag(b^{k_j}_{i-j-1}=)
\end{align}

\begin{align}
\notag&(d_{i-j}^{k_j-1}-d_{i-j-1}^{k_j})b_{i-j}^{k_j-1}\notag\\&=\prod^{k_{j+1}-3}_{e=k_{j}}(a^{e}_{m-i+j-e}-a^{e+1}_{m-i+j-e-1})(a^{k_{j+1}-2}_{m-i-k_{j+1}+j+2}-a^{k_{j+1}-1}_{m-i-k_{j+1}+j+1})b_{i-j-1}^{k_{j+1}-1}.\notag 
\end{align}

and relations [[AAC]]
\begin{align}
(c^{k_j-1}_{m-i-k_j+j}=)
\end{align}
\begin{align}
\notag&c_{m-i-k_{j-1}+j-1}^{k_{j-1}}\prod_{e=k_{j-1}}^{k_j-3}(a^{e}_{m-i+j-e-1}-a^{e+1}_{m-i+j-e-2})(a^{k_j-2}_{m-i-k_j+j+1}-a^{k_j-1}_{m-i-k_j+j})\notag\\&=(d_{i-j}^{k_j-1}-d_{i-j-1}^{k_j})c_{m-i-k_{j}+{j}}^{k_{j}}\notag
\end{align}
$for\; j\;=1,2,\ldots, l_1-2\;$  where $a^{k_{j}-1}_{m-i-k_j+j}$ ,\; $d^{k_j}_{i-j-1}$ for $ j=1,\ldots, l_1-2$ and $a^{k_{l_1-1}-1}_{m-i+k_{l_1}+l_1}$ be in the last page of the proof.
\begin{step}
$$A_1,\ldots,B_{l_1},\ldots, B_{l_2},\ldots,B_{l_t}$$
Let $z_i=k_{l_{i}-1}-k_{l_i}$ for $i=1,\ldots,t$.
$$ a_{m-i}^{0},c_{m-i}^{0},$$
$$a_{m-i-1}^{k_0+1},a_{m-i-2}^{k_0+2},\ldots, a_{m-i-k_1+2}^{k_1-2},$$
$$d_{i-1}^{k_1-1},b_{i-1}^{k_1-1}$$
$$a_{m-i-k_1+1}^{k_1},c_{m-i-k_1+1}^{k_1}$$
$$a_{m-i-k_1}^{k_1+1},\ldots, a_{m-i-k_2+3}^{k_2-2}$$
$$d_{i-2}^{k_2-1},b_{i-2}^{k_2-1}$$
$$a_{m-i-k_2+2}^{k_2}, c_{m-i-k_2+2}^{k_2}$$
$$a_{m-i-k_2+1}^{k_2+1},\ldots, a_{m-i-k_3+4}^{k_3-2},$$
$$d_{i-3}^{k_3-1}, b_{i-3}^{k_3-1}$$
$$\ldots$$
$$\ldots$$
$$d_{i-l_1+2}^{k_{l_1-2}-1},b_{i-l_1+2}^{k_ { l_1-2}-1}$$
$$a_{m-i-k_{ l_1-2}+ l_1-2}^{k_{l_1-2}},c_{m-i-k_{l_1-2}+ l_1-2}^{k_{l_1-2}}$$
$$a_{m-i-k_{ l_1-2}+l_1-3}^{k_{ l_1-2}+1},\ldots, a_{m-i-k_{ l_1-1}+l_1}^{k_{ l_1-1}-2}$$
$$d_{i-l_1+1}^{k_{l_1-1}-1}, b_{i-l_1+1}^{k_{l_1-1}-1}$$
\begin{center}
\textbf{$A_{l_1-1}----------------------------------------------$}
\end{center}
$$d_{i-l_1}^{k_{l_1-1}}, d_{i-l_1-1}^{k_{l_1-1}+1},\ldots, d_{i-l_1+z_1+2}^{k_{l_1}-2}$$
$$d_{i-l_1+z_1+1}^{k_{l_1}-1}$$
\begin{center}
\textbf{$B_{l_1}------------------------------------------$}
\end{center}
$$a^{k_{l_1}}_{m-i-k_{l_1}+l_1-z_1-1},c^{k_{l_1}}_{m-i-k_{l_1}+l_1-z_1-1},$$
$$a_{m-i-k_{l_1}+l_1-z_1-2}^{k_{l_1}+1},a_{m-i-k_{l_1}+l_1-z_1-3}^{k_{l_1}+2},\ldots,a_{m-i+l_1-k_{l_1+1}-z_1+1}^{k_{l_1+1}-2},$$
$$d_{i-l_1+z_1}^{k_{l_1+1}-1},b_{i-l_1+z_1}^{k_{l_1+1}-1}$$
$$a_{m-i-k_{l_1+1}-z_1+l_1}^{k_{l_1+1}},c_{m-i-k_{l_1+1}-z_1+l_1}^{k_{l_1+1}}$$
$$a_{m-i-k_1}^{k_{l_1+1}+1},\ldots, a_{m-i-k_2+3}^{k_{l_1+2}-2}$$
$$d_{i-2}^{k_{l_1+2}-1},b_{i-2}^{k_{l_1+2}-1}$$
$$a_{m-i-k_2+2}^{k_{l_1+2}}, c_{m-i-k_2+2}^{k_{l_1+2}}$$
$$a_{m-i-k_{l_1+2}+1}^{k_{l_1+2}+1},\ldots, a_{m-i-k_{l_1+3}+4}^{k_{l_1+3}-2},$$
$$d_{i-3}^{k_{l_1+3}-1}, b_{i-3}^{k_{l_1+3}-1}$$
$$a_{m-i-k_{l_1+3}+3}^{k_{l_1+3}},c_{m-i-k_{l_1+3}+3}^{k_{l_1+3}}$$
$$a_{m-i-k_{l_1+3}+2}^{k_{l_1+3}+1},\ldots, a_{m-i-k_{l_1+4}+5}^{k_{l_1+4}-2}$$
$$\ldots$$
$$\ldots$$
$$d_{i-l_2+z_1+1}^{k_{{l_{2}}-2}-1},b_{i-l_2+z_1+1}^{k_{{l_{2}}-2}-1}$$
$$a_{m-i+l_2-z_1-k_{l_2-2}-3}^{k_{{l_2}-2}},c_{m-i+l_2-z_1-k_{l_2-2}-3}^{k_{{l_2-2}}}$$
$$a_{m-i+l_2-z_1-k_{l_2-2}-4}^{k_{{l_{2}}-2}+1},\ldots, a_{m-i+l_2-z_1-k_{l_2-1}-1}^{k_{l_{2}-1}-2}$$
$$d_{i-l_2+z_1+2}^{k_{{l_{2}}-1}-1}, b_{i-l_2+z_1+2}^{k_{{l_{2}}-1}-1}$$
\textbf{$A_{l_2-1}----------------------------------------$}
$$\ldots$$
$$\ldots$$
\begin{center}
\textbf{$B_{l_{t-1}}------------------------------------------$}
\end{center}
$$ a_{m-k_{l_{t-1}}-i+l_{t-1}-\sum^{t-1}_{i=1}z_i-t+1}^{k_{l_{t-1}}},c_{m-k_{l_{t-1}}-i+l_{t-1}-\sum^{t-1}_{i=1}z_i-t+1}^{k_{l_{t-1}}},$$
$$a_{m-k_{l_{t-1}}-i+l_{t-1}-\sum^{t-1}_{i=1}z_i-t}^{k_{l_{t-1}}+1},a_{m-k_{l_{t-1}}-i+l_{t-1}-\sum^{t-1}_{i=1}z_i-t-1}^{k_{l_{t-1}}+2},\ldots, a_{m-i+l_{t-1}-\sum^{t-1}_{i=1}z_i-t+3-k_{l_{t-1}+1}}^{k_{l_{t-1}+1}-2},$$
$$d_{i-{l_{t-1}}+\sum z_i^{t-1}+t-2}^{k_{l_{t-1}+1}-1},b_{i-{l_{t-1}}+\sum z_i^{t-1}+t-2}^{k_{l_{t-1}+1}-1}$$
$$a_{m-k_{l_{t-1}+1}-i+l_{t-1}-\sum^{t-1}_{i=1}z_i-t+2}^{k_{l_{t-1}+1}},c_{m-k_{l_{t-1}+1}-i+l_{t-1}-\sum^{t-1}_{i=1}z_i-t+2}^{k_{l_{t-1}+1}}$$
$$\ldots$$
$$\ldots$$
$$d_{i-l_{t}+\sum^{t-1}_{i=1} z_i+t-1}^{k_{ l_t-2}-1},b_{i-l_{t}+\sum^{t-1}_{i=1} z_i+t-1}^{k_ {l_t-2}-1}$$
$$a_{m-i-k_{l_t-2}-\sum^{t-1}_{i=1} z_i+l_t-t-1}^{k_{l_t-2}},c_{m-i-k_{l l_t-2}-\sum^{t-1}_{i=1} z_i+l_t-t-1}^{k_{l_t-2}}$$
$$a_{m-i-k_{l_t-2}-\sum^{t-1}_{i=1} z_i+l_t-t-2}^{k_{l_t-2}+1},\ldots, a_{m-i-k_{ l_t-1}-\sum^{t-1}_{i=1} z_i+l_t-t+1}^{k_{l_t-1}-2}$$
$$d_{i-l_{t}+\sum^{t-1}_{i=1} z_i+t}^{k_{l_t-1}-1}, b_{i-l_{t}+\sum^{t-1}_{i=1} z_i+t}^{k_{l_t-1}-1}$$
\begin{center}
\textbf{$A_{l_t-1}-------------------------------------------$}
\end{center}
$$d_{i-l_{t}+\sum^{t-1}_{i=1} z_i+t-1}^{k_{l_t-1}}, d_{i-l_{t}+\sum^{t-1}_{i=1} z_i+t-2}^{k_{l_t-1}+1}\ldots, d_{i-l_{t}+\sum^{t}_{i=1} z_i+t+1}^{k_{l_t}-2}$$
$$a_1^{k_{l_t}-1}, \ldots, a^{k_{l_t}-1}_{m-k_{l_t}-i+l_t-\sum_{i=1}^{t-1} z_i-t},d_1^{k_{l_t}-1},\ldots, d_{i-l_{t}+\sum^{t}_{i=1} z_i+t}^{k_{l_t}-1}, c^{k_{l_t}-1}_{m-k_{l_t}-i+l_t-\sum_{i=1}^{t-1} z_i-t}$$
\begin{center}
\textbf{$B_{l_t}----------------------------------------------$}
\end{center}

$\text{Hilb}_{m_{\bullet}}(\tilde{R}/B)$ at the point  $$A_1..B_{l_1}.. B_{l_2}\ldots B_{l_t}$$
is defined by above parameters 
with equations [[ABC]]
\begin{align}
\notag(c^{k_{l_j-1}-1}_{m-k_{l_j-1}-i-\sum_{i=1}^{j-1} z_i+l_j-t}=) \;for\; j=1,2,\ldots, t.
\end{align}\small
\begin{align}
\notag&c^{k_{l_j}}_{m-i-k_{l_j}+l_j-\sum^{j}_{i=1} z_i-j}\left\{ \prod^{k_{l_j}-1}_{e=k_{l_{j}-1}}\left(d_{i-l_j+\sum^{j-1}_{i=1}z_i+j-e-1+k_{l_{j}-1}}^{e-1}-d_{i-l_j+\sum^{j-1}_{i=1}z_i+j-e-2+k_{l_{j}-1}}^{e}\right)\right\}=\notag\\
&\left\{\prod^{k_{l_j-1}-2}_{e=k_{l_j-2}}\left(a_{m-i-\sum^{j-1}_{i=1} z_i+l_j-j-1-e}^{e}-a_{m-i-\sum^{j-1}_{i=1} z_i+l_j-j-2-e}^{e+1}\right)\right\}c_{m-i-\sum^{j-1}_{i=1} z_i+l_j-j-1-k_{l_j-2}}^{k_{l_j-2}}\notag
\end{align}
with
$$a^{k_{l_j-1}-1}_{m-i-k_{l_j-1}-\sum^{j-1}_{i=1}z_i+l_j-j+1}=a^{k_{l_j}-1}_{m-i-k_{l_j-1}-\sum^{j-1}_{i=1}z_i+l_j-j+1}$$ $$+ \sum_{p=k_{l_j-1}}^{k_{l_j}-1} \prod_{q\geqslant k_{l_j-1}, q\neq p}^{k_{l_j}-1}(d^{q-1}_{i-l_j+\sum^{j-1}_{i=1}z_i+j-q+k_{l_j-1}}-d^{q}_{i-l_j+\sum^{j-1}_{i=1}z_i+j-q+k_{l_j-1}-1})b^{k_{l_j-1}-1}_{i-1}c^{k_{l_j}-1}_{m-i-k+1}$$
where $j=1,2,\ldots, t.$
with equations [[BAB]]
\begin{align}
\notag(b^{k_{l_j}}_{i+\sum^{j}_{p=1} z_i-l_j+j}=) \;for \; j=1,2,\ldots, t-1
\end{align}
\begin{align}
\notag&b^{k_{l_j-1}-1}_{i+\sum^{j}_{p=1} z_i-l_j+j-1}\left\{ \prod^{k_{l_j}}_{e=k_{l_{j}-1}}\left(d_{i-l_j+\sum^{j-1}_{i=1}z_i+j-e-1+k_{l_{j}-1}}^{e-1}-d_{i-l_j+\sum^{j-1}_{i=1}z_i+j-e-2+k_{l_{j}-1}}^{e}\right)\right\}\notag\\
&=b^{k_{l_j+1}-1}_{i+\sum^{j}_{p=1} z_i-l_j+j}\prod^{k_{l_j+1}-2}_{e=k_{l_j}}(a_{m-i+l_j-\sum^j_{i=1}z_i-j-e}^{e}-a_{m-i+l_j-\sum^j_{i=1}z_i-j-e-1}^{e+1})\notag
\end{align}

\begin{align}\notag(b^{k_{(l_{j}+s+1)}}_{i-{l_{j}}+\sum^j z_i+j-2-s}=)
\end{align}
with the relations [[SB]]
\begin{align}
\notag&(d_{i-{l_{j}}+\sum_{i=1}^j z_i+j-1-s}^{k_{(l_{j}+1+s)}-1}-d_{i-{l_{j}}+\sum^j z_i+j-2-s}^{k_{(l_{j}+1+s)}})b_{i-{l_{j}}+\sum^j z_i+j-1-s}^{k_{(l_{j}+1+s)}-1}\notag\\&=\left\{\prod^{k_{(l_j+2+s)}-2}_{e=k_{(l_{j}+1+s)}}\left(a^{e}_{m-i+l_j-\sum^j_{i=1} z_i-j+1+s-e}-a^{e+1}_{m-i+l_j-\sum^j_{i=1} z_i-j+s-e}\right)\right\}b_{i-{l_{j}}+\sum^j z_i+j-2-s}^{k_{(l_{j}+2+s)}-1}.\notag 
\end{align}
with $l_0=0$ , $s=0,..,l_{j+1}-l_{j}-3$ and $j=0,\ldots,t-1$   


and the relations[[AAC]]
\small
\begin{align}
\notag(c^{k_{(l_{j}+s+1)}-1}_{m-i+l_j-\sum^j_{i=1} z_i-j+s-k_{(l_j+1+s)}+1}=)
\end{align}
\begin{align}
&\left\{\prod^{k_{(l_j+1+s)}-2}_{e=k_{(l_j+s)}}\left(a^{e}_{m-i+l_j-\sum^j_{i=1} z_i-j+s-e}-a^{e+1}_{m-i+l_j-\sum^j_{i=1} z_i-j+s-e-1}\right)\right\}c_{m-i-k_{(l_j+s)}+l_j-\sum^j_{i=1} z_i-j+s-e}^{k_{(l_{j}+s)}}\notag\\&=(d_{i-{l_{j}}+\sum_{i=1}^j z_i+j-1-s}^{k_{(l_{j}+1+s
)}-1}-d_{i-{l_{j}}+\sum^j z_i+j-2-s}^{k_{(l_{j}+1+s)}})c_{m-i+l_j-\sum^j_{i=1} z_i-j+s-k_{(l_j+1+s)}+1}^{k_{(l_{j}+1+s)}}\notag
\end{align}

with $l_0=0$ , $s=0,..,l_{j+1}-l_{j}-3$ and $j=0,\ldots,t-1$   
where\small
$$a^{k_{(l_j+1+s)}-1}_{m-i+l_j-\sum^j_{i=1} z_i-j+s-k_{(l_j+1+s)}}=$$ $$a^{k_{(l_j+1+s)}}_{m-i+l_j-\sum^j_{i=1} z_i-j+s-k_{(l_j+1+s)}}+b^{k_{(l_j+1+s)}-1}_{i-l_j+\sum^j_{i=1}z_i+j-s}c^{k_{(l_j+1+s)}}_{m-i+l_j-\sum^j_{i=1} z_i-j+s-k_{(l_j+1+s)}}\;$$ $ when\; s=0,..,l_{j+1}-l_{j}-3\; and \;j=0,\ldots,t-1$ 
 and 

$$d^{k_{l_j+s+1}}_{i-l_j+\sum^j_{i=1}z_i+j-s-1}=d^{k_{l_j+s+1}-1}_{i-l_j+\sum^j_{i=1}z_i+j-s}-$$ $$(k_{l_j+s+2}-k_{l_j+s+1}-1)\prod^{k_{l_j+s+2}-3}_{e=k_{l_j+s+1}}(a^{e}_{m-i+j-e}-a^{e+1}_{m-i+j-e-1})(a^{k_{l_j+s+2}-2}_{m-i-k_{j+1}+j+2}-a^{k_{l_j+s+2}-1}_{m-i-k_{j+1}+j+1})b^{k_{l_j+s+2}-1}_{{i-l_j+\sum^j_{i=1}z_i+j-s}-1}$$
$ when\; s=0,..,l_{j+1}-l_{j}-3\; and \;j=0,\ldots,t-1.$ 
\end{step}

\begin{step}
$$A_1..B_{l_1}.. B_{l_2}\ldots B_{l_t}A_{l_{t+1}}$$
Let $z_i=k_{l_{i}-1}-k_{l_i}$ for $i=1,\ldots, t$.
$$ a_{m-i}^{0},c_{m-i}^{0},$$
$$a_{m-i-1}^{k_0+1},a_{m-i-2}^{k_0+2},\ldots,a_{m-i-k_1+2}^{k_1-2},$$
$$d_{i-1}^{k_1-1},b_{i-1}^{k_1-1}$$
$$a_{m-i-k_1+1}^{k_1},c_{m-i-k_1+1}^{k_1}$$
$$a_{m-i-k_1}^{k_1+1},\ldots, a_{m-i-k_2+3}^{k_2-2}$$
$$d_{i-2}^{k_2-1},b_{i-2}^{k_2-1}$$
$$a_{m-i-k_2+2}^{k_2}, c_{m-i-k_2+2}^{k_2}$$
$$a_{m-i-k_2+1}^{k_2+1},\ldots, a_{m-i-k_3+4}^{k_3-2},$$
$$d_{i-3}^{k_3-1}, b_{i-3}^{k_3-1}$$
$$\ldots$$
$$\ldots$$
$$d_{i-l_1+2}^{k_{ l_1-2}-1},b_{i-l_1+2}^{k_ {l_1-2}-1}$$
$$a_{m-i-k_{ l_1-2}+ l_1-2}^{k_{ l_1-2}},c_{m-i-k_{l_1-2}+ l_1-2}^{k_{ l_1-2}}$$
$$a_{m-i-k_{ l_1-2}+l_1-3}^{k_{l_1-2}+1},\ldots, a _{m-i-k_{l_1-1}+l_1}^{k_{l_1-1}-2}$$

$$d_{i-l_1+1}^{k_{l_1-1}-1}, b_{i-l_1+1}^{k_{l_1-1}-1}$$
\begin{center}
\textbf{$A_{l_1-1}-------------------------------------------$}
\end{center}
$$d_{i-l_1}^{k_{l_1-1}}, d_{i-l_1-1}^{k_{l_1-1}+1},\ldots, d_{i-l_1+z_1+2}^{k_{l_1}-2}$$
$$d_{i-l_1+z_1+1}^{k_{l_1}-1}$$
\begin{center}
\textbf{$B_{l_1}----------------------------------------------$}
\end{center}
$$a^{k_{l_1}}_{m-i-k_{l_1}+l_1-z_1-1},c^{k_{l_1}}_{m-i-k_{l_1}+l_1-z_1-1},$$
$$a_{m-i-k_{l_1}+l_1-z_1-2}^{k_{l_1}+1},a_{m-i-k_{l_1}+l_1-z_1-3}^{k_{l_1}+2},\ldots,a_{m-i+l_1-k_{l_1+1}-z_1+1}^{k_{l_1+1}-2},$$
$$d_{i-l_1+z_1}^{k_{l_1+1}-1},b_{i-l_1+z_1}^{k_{l_1+1}-1}$$
$$a_{m-i-k_{l_1+1}-z_1+l_1}^{k_{l_1+1}},c_{m-i-k_{l_1+1}-z_1+l_1}^{k_{l_1+1}}$$
$$a_{m-i-k_1}^{k_{l_1+1}+1},\ldots, a_{m-i-k_2+3}^{k_{l_1+2}-2}$$
$$d_{i-2}^{k_{l_1+2}-1},b_{i-2}^{k_{l_1+2}-1}$$
$$a_{m-i-k_2+2}^{k_{l_1+2}}, c_{m-i-k_2+2}^{k_{l_1+2}}$$
$$a_{m-i-k_{l_1+2}+1}^{k_{l_1+2}+1},\ldots, a_{m-i-k_{l_1+3}+4}^{k_{l_1+3}-2},$$
$$d_{i-3}^{k_{l_1+3}-1}, b_{i-3}^{k_{l_1+3}-1}$$
$$a_{m-i-k_{l_1+3}+3}^{k_{l_1+3}},c_{m-i-k_{l_1+3}+3}^{k_{l_1+3}}$$
$$a_{m-i-k_{l_1+3}+2}^{k_{l_1+3}+1},\ldots, a_{m-i-k_{l_1+4}+5}^{k_{l_1+4}-2}$$
$$\ldots$$
$$\ldots$$
$$d_{i-l_2+z_1+1}^{k_{{l_{2}}-2}-1},b_{i-l_2+z_1+1}^{k_{{l_{2}}-2}-1}$$
$$a_{m-i+l_2-z_1-k_{l_2-2}-3}^{k_{{l_2}-2}},c_{m-i+l_2-z_1-k_{l_2-2}-3}^{k_{{l_2-2}}}$$
$$a_{m-i+l_2-z_1-k_{l_2-2}-4}^{k_{{l_{2}}-2}+1},\ldots, a_{m-i+l_2-z_1-k_{l_2-1}-1}^{k_{l_{2}-1}-2}$$
$$d_{i-l_2+z_1+2}^{k_{{l_{2}}-1}-1}, b_{i-l_2+z_1+2}^{k_{{l_{2}}-1}-1}$$
\begin{center}
\textbf{$A_{l_2-1}----------------------------------------------$}
\end{center}
$$\ldots$$
$$\ldots$$

\begin{center}
\textbf{$B_{l_{t-1}}----------------------------------------------$}
\end{center}
$$ a_{m-k_{l_{t-1}}-i+l_{t-1}-\sum^{t-1}_{i=1}z_i-t+1}^{k_{l_{t-1}}},c_{m-k_{l_{t-1}}-i+l_{t-1}-\sum^{t-1}_{i=1}z_i-t+1}^{k_{l_{t-1}}},$$
$$a_{m-k_{l_{t-1}}-i+l_{t-1}-\sum^{t-1}_{i=1}z_i-t}^{k_{l_{t-1}}+1},a_{m-k_{l_{t-1}}-i+l_{t-1}-\sum^{t-1}_{i=1}z_i-t-1}^{k_{l_{t-1}}+2},\ldots,a_{m-i+l_{t-1}-\sum^{t-1}_{i=1}z_i-t+3-k_{l_{t-1}+1
}}^{k_{l_{t-1}+1}-2},$$
$$d_{i-{l_{t-1}}+\sum z_i^{t-1}+t-2}^{k_{l_{t-1}+1}-1},b_{i-{l_{t-1}}+\sum z_i^{t-1}+t-2}^{k_{l_{t-1}+1}-1}$$
$$a_{m-k_{l_{t-1}+1}-i+l_{t-1}-\sum^{t-1}_{i=1}z_i-t+2}^{k_{l_{t-1}+1}},c_{m-k_{l_{t-1}+1}-i+l_{t-1}-\sum^{t-1}_{i=1}z_i-t+2}^{k_{l_{t-1}+1}}$$
$$a_{m-i-k_1}^{k_1+1},\ldots, a_{m-i-k_2+3}^{k_2-2}$$
$$d_{i-2}^{k_2-1},b_{i-2}^{k_2-1}$$
$$a_{m-i-k_2+2}^{k_2}, c_{m-i-k_2+2}^{k_2}$$
$$a_{m-i-k_2+1}^{k_2+1},\ldots, a_{m-i-k_3+4}^{k_3-2},$$
$$d_{i-3}^{k_3-1}, b_{i-3}^{k_3-1}$$
$$\ldots$$
$$\ldots$$
$$d_{i-l_{t}+\sum^{t-1}_{i=1} z_i+t-1}^{k_{ l_t-2}-1},b_{i-l_{t}+\sum^{t-1}_{i=1} z_i+t-1}^{k_ { l_t-2}-1}$$
$$a_{m-i-k_{l_t-2}-\sum^{t-1}_{i=1} z_i+l_t-t-1}^{k_{ l_t-2}},c_{m-i-k_{ l_t-2}-\sum^{t-1}_{i=1} z_i+l_t-t-1}^{k_{l_t-2}}$$
$$a_{m-i-k_{l_t-2}-\sum^{t-1}_{i=1} z_i+l_t-t-2}^{k_{l_t-2}+1},\ldots, a _{m-i-k_{l_t-1}-\sum^{t-1}_{i=1} z_i+l_t-t+1}^{k_{l_t-1}-2}$$

$$d_{i-l_{t}+\sum^{t-1}_{i=1} z_i+t}^{k_{l_t-1}-1}, b_{i-l_{t}+\sum^{t-1}_{i=1} z_i+t}^{k_{l_t-1}-1}$$
\begin{center}
\textbf{$A_{l_t-1}----------------------------------------------$}
\end{center}

$$d_{i-l_{t}+\sum^{t-1}_{i=1} z_i+t-1}^{k_{l_t-1}}, d_{i-l_{t}+\sum^{t-1}_{i=1} z_i+t-2}^{k_{l_t-1}+1}\ldots, d_{i-l_{t}+\sum^{t}_{i=1} z_i+t+1}^{k_{l_t}-2}$$
$$d_{i-l_{t}+\sum^{t}_{i=1} z_i+t}^{k_{l_t}-1}$$
\begin{center}
\textbf{$B_{l_t}----------------------------------------------$}
\end{center}
$$a^{k_{l_t}}_{m-i-k_{l_t}+l_t-\sum^{t}_{i=1} z_i-t},c^{k_{l_t}}_{m-i-k_{l_t}+l_t-\sum^{t}_{i=1} z_i-t},$$
$$a_{m-i-k_{l_t}+l_t-\sum^{t-1}_{i=1} z_i-t-1}^{k_{l_t}+1},a_{m-i-k_{l_t}+l_t-\sum^{t}_{i=1} z_i-t-2}^{k_{l_t}+2},\ldots,a_{m-i-k_{l_t+1}+l_t-\sum^{t}_{i=1} z_i-t+2}^{k_{l_t+1}-2},$$
$$d_{i-l_t+\sum^t_{i=1}z_i+t-1}^{k_{l_t+1}-1},b_{i-l_t+\sum^t_{i=1}z_i+t-1}^{k_{l_t+1}-1}$$
$$a_{m-i-k_{l_t+1}+l_t-\sum^t_{i=1} z_i-t+1}^{k_{l_t+1}},c_{m-i-k_{l_t+1}+l_t-\sum^t_{i=1} z_i-t+1}^{k_{l_t+1}}$$
$$a_{m-i-k_1}^{k_{l_1+1}+1},\ldots, a_{m-i-k_2+3}^{k_{l_1+2}-2}$$
$$d_{i-2}^{k_{l_1+2}-1},b_{i-2}^{k_{l_1+2}-1}$$
$$a_{m-i-k_2+2}^{k_{l_1+2}}, c_{m-i-k_2+2}^{k_{l_1+2}}$$
$$a_{m-i-k_{l_1+2}+1}^{k_{l_1+2}+1},\ldots, a_{m-i-k_{l_1+3}+4}^{k_{l_1+3}-2},$$
$$d_{i-3}^{k_{l_1+3}-1}, b_{i-3}^{k_{l_1+3}-1}$$
$$a_{m-i-k_{l_1+3}+3}^{k_{l_1+3}},c_{m-i-k_{l_1+3}+3}^{k_{l_1+3}}$$
$$a_{m-i-k_{l_1+3}+2}^{k_{l_1+3}+1},\ldots, a_{m-i-k_{l_1+4}+5}^{k_{l_1+4}-2}$$
$$\ldots$$
$$\ldots$$
$$d_{i-{l_{t+1}}+\sum^t_{i=1} z_i+t+1}^{k_{{l_{t+1}}-1}-1},b_{i-{l_{t+1}}+\sum^t_{i=1} z_i+t+1}^{k_{{l_{t+1}}-1}-1}$$
$$a_{m-i-k_{l_{t+1}-1}+{l_{t+1}}-\sum^t z_i-t-1}^{k_{{l_{t+1}}-1}},c_{m-i-k_{l_{t+1}-1}+{l_{t+1}}-\sum^t z_i-t-1}^{k_{{l_{t+1}}-1}}$$
$$a_{m-i-k_{l_{t+1}-1}+{l_{t+1}}-\sum^t z_i-t-2}^{k_{{l_{t+1}}-1}+1},\ldots, a _{m-i+l_{t+1}-\sum^tz_i-t-k_{l_{t+1}}+1}^{k_{l_{t+1}}-2}$$
$$a_1^{k_{l_{t+1}}-1},\ldots, a_{m-i+l_{t+1}-\sum^tz_i-t-k_{l_{t+1}}}^{k_{l_{t+1}}-1}, d_1^{k_{l_{t+1}}-1},\ldots,d_{i-{l_{t+1}}+\sum^t_{i=1}z_i+t}^{k_{l_{t+1}}-1}, b_{i-{l_{t+1}}+\sum^t_{i=1}z_i+t}^{k_{{l_{t+1}}}-1}$$
\begin{center}
\textbf{$A_{l_{t+1}}----------------------------\;\;\;with\;l_{t+1}=h\;$}
\end{center}
$\text{Hilb}_{m_{\bullet}}(\tilde{R}/B)$ at the point  $$A_1..B_{l_1}.. B_{l_2}\ldots.B_{l_t}A_{l_{t+1}}$$
is defined by above parameters 
with equations [[ABC]]\small
\begin{align}
\notag(c^{k_{l_j-1}-1}_{m-k_{l_j-1}-i-\sum_{i=1}^{j-1} z_i+l_j-t}=) \;for\; j=1,2,\ldots, t
\end{align}
\begin{align}
\notag&c^{k_{l_j}}_{m-i-k_{l_j}+l_j-\sum^{j}_{i=1} z_i-j}\left\{ \prod^{k_{l_j}-1}_{e=k_{l_{j}-1}}\left(d_{i-l_j+\sum^{j-1}_{i=1}z_i+j-e-1+k_{l_{j}-1}}^{e-1}-d_{i-l_j+\sum^{j-1}_{i=1}z_i+j-e-2+k_{l_{j}-1}}^{e}\right)\right\}\notag\\
&=\left\{\prod^{k_{l_j-1}-2}_{e=k_{l_j-2}}\left(a_{m-i-\sum^{j-1}_{i=1} z_i+l_j-j-1-e}^{e}-a_{m-i-\sum^{j-1}_{i=1} z_i+l_j-j-2-e}^{e+1}\right)\right\}c_{m-i-\sum^{j-1}_{i=1} z_i+l_j-j-1-k_{l_j-2}}^{k_{l_j-2}}\notag
\end{align}

with equations [[BAB]]
\begin{align}
\notag(b^{k_{l_j}}_{i+\sum^{j}_{p=1} z_i-l_j+j-1}=) \;for \; j=1,2,\ldots, t
\end{align}
\begin{align}
\notag&b^{k_{l_j-1}-1}_{i+\sum^{j-1}_{p=1} z_i-l_j+j-1}\left\{ \prod^{k_{l_j}}_{e=k_{l_{j}-1}}\left(d_{i-l_j+\sum^{j-1}_{i=1}z_i+j-e-1+k_{l_{j}-1}}^{e-1}-d_{i-l_j+\sum^{j-1}_{i=1}z_i+j-e-2+k_{l_{j}-1}}^{e}\right)\right\}\notag\\
&=b^{k_{l_j+1}-1}_{i+\sum^{j}_{p=1} z_i-l_j+j-1}\prod^{k_{l_j+1}-2}_{
e=k_{l_j}}(a_{m-i+l_j-\sum^j_{i=1}z_i-j}^{e}-a_{m-i+l_j-\sum^j_{i=1}z_i-j-e-1}^{e+1})\notag
\end{align}

\begin{align}\notag(b^{k_{(l_{j}+s+1)}}_{i-{l_{j}}+\sum^j z_i+j-2-s}=)
\end{align}
with the relations [[AAB]]
\begin{align}
\notag&(d_{i-{l_{j}}+\sum_{i=1}^j z_i+j-1-s}^{k_{(l_{j}+1+s)}-1}-d_{i-{l_{j}}+\sum^j z_i+j-2-s}^{k_{(l_{j}+1+s)}})b_{i-{l_{j}}+\sum^j z_i+j-1-s}^{k_{(l_{j}+1+s)}-1}\notag\\&=\left\{\prod^{k_{(l_j+2+s)}-2}_{e=k_{(l_{j}+1+s)}}\left(a^{e}_{m-i+l_j-\sum^j_{i=1} z_i-j+1+s-e}-a^{e+1}_{m-i+l_j-\sum^j_{i=1} z_i-j+s-e}\right)\right\}b_{i-{l_{j}}+\sum^j z_i+j-2-s}^{k_{(l_{j}+2+s)}-1}.\notag 
\end{align}
with $l_0=0$ , $s=0,..,l_{j+1}-l_{j}-3$ and $j=0,\ldots,t$   

and the relations[[AAC]]

\begin{align}
\notag(c^{k_{(l_{j}+s+1)}-1}_{m-i+l_j-\sum^j_{i=1} z_i-j+s-k_{(l_j+1+s)}+1}=)
\end{align}
\begin{align}
&\left\{\prod^{k_{(l_j+1+s)}-2}_{e=k_{(l_j+s)}}\left(a^{e}_{m-i+l_j-\sum^j_{i=1} z_i-j+s-e}-a^{e+1}_{m-i+l_j-\sum^j_{i=1} z_i-j+s-e-1}\right)\right\}c_{m-i-k_{(l_j+s)}+l_j-\sum^j_{i=1} z_i-j+s}^{k_{(l_{j}+s)}}\notag\\&=(d_{i-{l_{j}}+\sum_{i=1}^j z_i+j-1-s}^{k_{(l_{j}+1+s)}-1}-d_{i-{l_{j}}+\sum^j z_i+j-2-s}^{k_{(l_{j}+1+s)}})c_{m-i+l_j-\sum^j_{i=1} z_i-j+s-k_{(l_j+1+s)}+1}^{k_{(l_{j}+1+s)}}\notag
\end{align}
with $l_0=0$ , $s=0,..,l_{j+1}-l_{j}-3$ and $j=0,\ldots,t$ where

$$a^{k_{(l_j+1+s)}-1}_{m-i+l_j-\sum^j_{i=1} z_i-j+s-k_{(l_j+1+s)}}=$$
$$a^{k_{(l_j+1+s)}}_{m-i+l_j-\sum^j_{i=1} z_i-j+s-k_{(l_j+1+s)}}+b^{k_{(l_j+1+s)}-1}_{i-l_j+\sum^j_{i=1}z_i+j-s}c^{k_{(l_j+1+s)}}_{m-i+l_j-\sum^j_{i=1} z_i-j+s-k_{(l_j+1+s)}}\;$$ $ when\; s=0,..,l_{j+1}-l_{j}-3\; and \;j=0,\ldots,t$

$$d^{k_{l_j+s+1}}_{i-l_j+\sum^j_{i=1}z_i+j-s-1}=d^{k_{l_j+s+1}-1}_{i-l_j+\sum^j_{i=1}z_i+j-s}+$$ $$\left(\sum^{p=k_{l_j+s+1}}_{k_{l_j+s+2}-3}\prod^{k_{l_j+s+2}-3}_{e\geqslant k_{l_j+s+1}, e\neq p}(a^{e}_{m-i+j-e}-a^{e+1}_{m-i+j-e-1})(a^{k_{l_j+s+2}-2}_{m-i-k_{j+1}+j+2}-a^{k_{l_j+s+2}-1}_{m-i-k_{j+1}+j+1})\right)b^{k_{l_j+s+2}-1}_{{i-l_j+\sum^j_{i=1}z_i+j-s}-1}$$
$ when\; s=0,..,l_{j+1}-l_{j}-3\; and \;j=0,\ldots,t.$ 
and 
$$a^{k_{l_j-1}-1}_{m-i-k_{l_j-1}-\sum^{j-1}_{i=1}z_i+l_j-j+1}=a^{k_{l_j}-1}_{m-i-k_{l_j-1}-\sum^{j-1}_{i=1}z_i+l_j-j+1}+$$ $$ \sum_{p=k_{l_j-1}}^{k_{l_j}-1} \prod_{q\geqslant k_{l_j-1}, q\neq p}^{k_{l_j}-1}(d^{q-1}_{i-l_j+\sum^{j-1}_{i=1}z_i+j-q+k_{l_j-1}}-d^{q}_{i-l_j+\sum^{j-1}_{i=1}z_i+j-q+k_{l_j-1}-1})b^{k_{l_j-1}-1}_{i-1}c^{k_{l_j}-1}_{m-i-k+1}$$
where $j=1,2,\ldots, t.$
\end{step}




\end{step}
\end{proof}


\section{ Puctual \text{Hilb}ert scheme of points on the cusp curve.}

\subsection{As a set}

 Let R be $(\mathbb{C}[[x,y]]/(x^2-y^3))_{(x,y)}(=\mathbb{C}[[x,y]]/(x^2-y^3))$.
\begin{lemma}
Every non unit element in R can be associated to $xy^m+ay^n$, $\;y^s, \;xy^s\;or \;x$ for  some integers $0\leqslant m<n, 1\leqslant s$ and some nonzero constants a, b. 
 \end{lemma}
 \begin{proof}
 Let f and g be in R.
 We define f is associated to g if there exist unit $u\in R$ such that $f=ug.$
 For any $f\in R$, since $x^2=y^3$,  we can say $f(x,y)=x(h_1(y))+h_2(y)+a_0$ where  $a_0\in \mathbb{C}$ and 
$h_1(y)$ and $h_2(y)$ are power series in y with $deg(h_1(y))\geqslant 0,\;deg(h_2(y))\geqslant 1.$
 We may assume $a_0=0$ since $a_0\neq0$ implies that f is unit.
  Let $degh_1(y)=m\geqslant 0$ and $degh_2(y)=n\geqslant 1$.
 For the case $h_1(y)=0$, then $f=y^n(g_2(y))$ and $g_2(y)$ is unit.
 Thus f can be associated to $y^n.$
For the case $h_1(y)\neq 0$, we can rewrite $f=xy^m(g_1(y))+y^n(g_2(y))$ where $g_1(y)\;,g_2(y)$ are units. Since $g_2(y)$ is unit,  we can get $u(xy^m)+cy^n,m\geqslant 0, n\geqslant 1.$ for some unit u and constant $c\neq 0$ which is associated to f.  For given $a_i', b_i',c \in \mathbb{C}$ with $c\neq 0$ and $b_0'\neq 0$,  let us show there exist $a_i, b_i ,d\in \mathbb{C}\;, \;d\neq 0$ which satisfy the following equation
$$(\sum_{i=0}^{\infty}b_i'y^i+\sum_{i=0}^{\infty}a_i'xy^{i})x+cy^n=(\sum_{i=0}^{\infty}b_iy^i+\sum_{i=0}^{\infty}a_ixy^{i})(x+dy^n).$$  
After comparing coefficients for $xy^i, i\geqslant 0$  we have following equations
$$b_0'=b_0, \;i=0$$
$$b_i'=b_i,\; 1\leqslant i\leqslant n-1$$
$$b_i'=b_i+a_{i-n}d,\; i\geqslant n$$ 
and coefficients for $y^i, i\geqslant 1,\; a_l=d_l=b_i=0\; when\; l<0$
$$a_{n-3}'+c=a_{n-3}+db_0, i= n$$
$$a_{i-3}'=a_{i-3}+db_{i-n}, i\neq n$$
For n=1, for given unit u and nonzero constant c, let us show that $ux+cy$ is associated to $x+dy$ for some nonzero constant d. For given $a_i', b_i',c \in \mathbb{C}$ with $c\neq 0$ and $b_0'\neq 0$, 
let us show there exist $a_i, b_i ,d\in \mathbb{C}\; with \;d\neq 0$ with
$$(\sum_{i=0}^{\infty}b_i'y^i+\sum_{i=0}^{\infty}a_i'xy^{i})x+cy=(\sum_{i=0}^{\infty}b_iy^i+\sum_{i=0}^{\infty}a_ixy^{i})(x+dy).$$  
 The equations for coefficients of $xy^i, i\geqslant 0$ are 
$$b_0'=b_0\;, \;i=0$$
$$b_i'=b_i+da_{i-1}\;,\; i\geqslant 1$$
The coefficients for $y^i,\; i\geqslant 1$ of the equation are
$$c=db_0\;,i=1\;$$
$$0=db_1\;,i=2\;$$
$$a_{i-3}'=db_{i-1}+a_{i-3}, i\geqslant 3$$
We can show that there exist $a_i, b_i\; and\; d\in \mathbb{C}\; with\; d\neq 0$.\\ This infinite linear systems have a solution(exactly one solution.) Chasing it from $b_0'=b_0$. Note $b_0\neq 0$ since by assumption $b_0'\neq 0$. Since we have $c\neq 0$ by assumption and $c=db_0$, $d\neq 0$. In fact, $d=c/b_0.$ 
Since $0=db_1$ and $d\neq 0$ we have $b_1=0$. Hence the equation $b_1'=b_1+da_0$ implies $b_1'=da_0$. That is, $a_0=b_1'b_0/c.$
For this equation $a_0'=db_2+a_0$, we know $a_0', d, and a_0.$ Thus we can solve $b_2$ which is uniquely determined as $b_2=(a_0'-b_1'b_0/c)/(c/b_0)$. Similarly, $b_k$ determines $a_{k-1}$ and $a_{k-1}$ determines $b_{k+1}$ uniquely for $k\geqslant 2$. In such a way, we can find all $a_i$ and $b_i$ and d.\\
For n=2, coefficients for $xy^i, i\geqslant 0,\; a_l=b_l=0\; when\; l<0$
$$b_0'=b_0, \;i=0$$
$$b_i'=b_i,\; i= 1$$
$$b_i'=b_i+a_{i-2}d,\; i\geqslant 2$$ 
and coefficients for $y^i, i\geqslant 1,\; a_l=d_l=b_i=0\; when\; l<0$
$$c=db_0, i= 2$$
$$a_{i-3}'=a_{i-3}+db_{i-2}, i\geqslant 3$$
For $n=3,$ coefficients for $xy^i, i\geqslant 0,\; a_l=b_l=0\; when\; l<0$
$$b_0'=b_0, \;i=0$$
 $$b_i'=b_i,\; 0\leqslant  i \leqslant n-1$$
$$b_i'=b_i+a_{i-n}d,\; i\geqslant n$$ 
and coefficients for $y^i, i\geqslant 1,\; a_l=d_l=b_i=0\; when\; l<0$
$$a_{0}'+c=a_{0}+db_0, i= 3$$
$$a_{i-3}'=a_{i-3}+db_{i-n}, i\geqslant 3+1$$
For $n> 3,$  coefficients for $xy^i, i\geqslant 0,\; a_l=b_l=0\; when\; l<0$
$$b_0'=b_0, \;i=0$$
 $$b_i'=b_i,\; 0\leqslant  i \leqslant n-1$$
$$b_i'=b_i+a_{i-n}d,\; i\geqslant n$$ 
and coefficients for $y^i, i\geqslant 1,\; a_l=d_l=b_i=0\; when\; l<0$
$$a_{i-3}'=a_{i-3}\;, 3\leqslant i\leqslant n-1$$
$$a_{n-3}'+c=a_{n-3}+db_0, i= n$$
$$a_{i-3}'=a_{i-3}+db_{i-n}, i\geqslant n+1$$
 Let u be  $(\sum_{i=0}^{\infty}b_i'y^i+\sum_{i=0}^{\infty}{a_i}'xy^i)xy^m+cy^n$
 with $a_i', b_i', c\neq 0\; and\; b_0'\neq 0\in \mathbb{C}$. For given unit u and nonzero constant c, let us show that $uxy^m+cy^m$ is associated to $xy^m+dy^m$ for some nonzero constant d.
 Let u be  $(\sum_{i=0}^{\infty}b_i'y^i+\sum_{i=0}^{\infty}{a_i}'xy^i)xy^m+cy^n$
 with $a_i', b_i',c \in \mathbb{C}$ with $c\neq 0$ and $b_0'\neq 0$.
 Let us show that there exist  $a_i ,b_i ,\; d\in \mathbb{C}$ with $b_0\neq 0$ and $d\neq 0$ satisfying
 $$(\sum_{i=0}^{\infty}b_i'y^i+\sum_{i=0}^{\infty}{a_i}'xy^i)xy^m+cy^n
 =(\sum_{i=0}^{\infty}a_iy^i+\sum_{i=0}^{\infty}a_ixy^i)(xy^m+dy^n).$$
For $m\geqslant n$, $uxy^m+cy^n=(uxy^{m-n}+c)y^n$ implies that it is associated $y^n.$ 
For the case $m<n$, let $n=m+l.$ 
 For $l=1$, coefficients of $xy^{m+i}, i\geqslant 0$ are
$$b_0'=b_0$$
$$b_1'=b_1+a_0$$
$$b_2'=b_2+da_1$$
$$b_i'+a_{i-3}'=b_i+da_{i-1}\;,\; i\geqslant3$$ 
coefficients for $y^{m+1+i}, i\geqslant 0$ are
$$c=db_0$$
$$0=db_1$$
$$a_{i-2}'=a_{i-2}+db_{i}\;,\; i\geqslant 2.$$
Let us show there exist solution $a_i, b_i , d, d\neq 0$ for the linear infinite systems above. Chasing from $b_0=b_0'$.
$b_0=b_0'$ and $b_0\neq 0$ since $b_0'\neq 0$.
Then we get $d\neq 0$, then $b_1=0$. Then we can get $a_0$\;, $b_2$\;, $a_1$\;, $b_3$\;, $a_2$ and $b_4$ in a order. Thus we can find solution for all $a_i\;, \;b_i$ and d with $b_0\neq 0$ and $d\neq 0.$
For $l=2$, coefficients of $xy^{m+i}, i\geqslant 0$ are
$$b_0'=b_0$$
$$b_1'=b_1$$
$$b_2'=b_2+a_0$$
$$b_i'=b_i+da_{i-2}\;,\;i\geqslant 3$$
and coefficients of $y^{m+2+i}, i\;\geqslant 0$
$$c=db_0$$
$$a_{i-1}'=a_{i-1}+db_{i}\;,\; i\geqslant 1$$
For this infinite linear systems we can find exactly solution.
We can get $b_0, b_1\;,\;b_0 \neq0$ first. Then we can determine $d$ with $d\neq 0$. Then we can get $a_0$. The next  one we can determine is $b_2$. Then $a_1$ is followed. In other words, $b_{i+2}$ is followed by $a_i$ for $i\geqslant 0$. $a_i$ is followed by $b_{i+1}$ for $i\geqslant 1$ 
Hence, we can find all solution $a_i, b_i ,d$ where $b_0\neq 0\; and \;d\neq 0.$
 For $l=3$, the equation of coefficients of $xy^{m+i},i\geqslant 0$ are
$$b_i'=b_i,\; 0\leqslant i\leqslant 2$$
$$b_i'=b_i+da_{i-3}\;,\;i\geqslant 3$$
and comparing coefficients of $y^{m+3+i},\; i\geqslant 0$
$$a_0'+c=b_0d+a_0$$
$$a_i'=db_i+a_i\;,\; i\geqslant 1$$
For this infinite linear systems  we can find only one solution.
We can get $b_0, b_1, b_2\;,\;b_0 \neq0$ first. Then we can determine $d, a_0,a_1,a_2$ with $d\neq 0$. Then we can get $b_3, b_4 ,b_5$. The next we can determine is $a_3, a_4, a_5$. Then $b_6,b_7,b_9$ is followed. In this way we can find all solution $a_i, b_i ,d$ where $b_0\neq 0\; and \;d\neq 0.$
 For $l>3$ the equations for coefficients of $xy^{m+i}\;,\; i\geqslant 0$
$$b_i'=b_i, 0\leqslant i \leqslant l-1$$
$$b_i'=b_i +a_{i-l}, i\geqslant l$$ 
and the equations for coefficients of $y^{m+i+3}\;,\; i\geqslant 0$
$$a_i'=a_i,\; 0\leqslant i \leqslant l-4$$ 
$$a_i'+c=a_i+db_0, \;i=l-3$$
$$a_i'=a_i+db_{i-l+3}, i\geqslant l-2$$ 
  $b_i=b_i', 0\leqslant l-1, b_0\neq 0$ and $a_0,..., a_{l-4}$ are determined.  
 Then we can choose $d=t$ for all $t\in \mathbb{C}$. Then from $a_i'+c=a_i+db_0, \;i=l-3$, $a_i'+c=a_i+tb_0', \;i=l-3$ we determine $a_{l-3}$ since we already know $b_0$. 
 Then the equations $a_i'=a_i+db_{i-l+3},  l-2\leqslant i\leqslant2l-4$ and  $b_i, \;0\leqslant l-1$ tell us $a_i$ for $l-2\leqslant i\leqslant2l-4$. Then  
 $b_i'=b_i +a_{i-l},\; l\leqslant i\leqslant 3l-4 $ and $a_i$ for $l-2\leqslant i\leqslant2l-4$ determine $b_i\;,\; l\leqslant i\leqslant 3l-4 $. Then $a_i,\; 2l-3\leqslant i\leqslant 4l-7$ follow. Then $b_i,\; 3l-3\leqslant 3l-4$ do.  There are infinite solutions depending on the choice of  d.  We can always choose $d\neq 0.$ \end{proof}
Now let us consider the ideal in R.  A nonzero ideal in R contains non units since a maximal ideal is (x,y) which is an unique maximal ideal.  If $m\geqslant n$, $xy^m+ay^n=y^n(xy^{m-n}+a)$ which is associated to $y^n.$
\begin{lemma}
Let I be an ideal $(xy^m+ay^n)$ for $0\leqslant m<n, a\neq0$. Then  I is one of $(xy^m+ay^{m+1})$, $(xy^m+ay^{m+2})$ or $(xy^m)$. 
\end{lemma}
\begin{proof}
Note that $xy^m=-ay^n$ in $R/I$ which implies $x^2y^m=-axy^n=-a(xy^m)y^{n-m}=-a(y^n)y^{n-m}=-ay^{2n-m}$ in $R/I$. Since $x^2=y^3$ and $x^2y^m=-ay^{2n-m}$ in $R/I$, $y^{m+3}+ay^{2n-m}\in I.$  Then $y^{min\{m+3,2n-m\}}\in I.$
For the case m+3=2n-m we have
2m+3=2n which is impossible.
If  $m+3<2n-m$, then we can consider two cases
$m+3<n<2n-m$ and $m<n<m+3<2n-m$. For $m+3<n<2n-m$,
$y^{m+3}\in I$, then $x(xy^m)y^{n-m-3}=y^n\in I$. Hence $I=(xy^m)$.  
For $m<n<m+3<2n-m$, we must have$n=m+2,$ thus $I=(xy^m+ay^{m+2}).$ 
If $2n-m<m+3$ we get
$n=m+1.$ Thus $I=(xy^m+ay^{m+1})$. 
\end{proof}

\begin{remark} Any ideal in R can be generated by  $xy^m+ay^{m+1},\;xy^m+ay^{m+2}$$(TypeII),\;y^m,\;x+ay^2,\;xy^m(TypeI)$ for $0\leqslant m\in \mathbb{N}$. 
\end{remark}

\begin{theorem}
 An ideal in R is one of ideals \\
 (1) $(y^m), (x)(=(x,y^3)) , (xy^{m}), (xy^m,y^{m+1}),  (xy^m,y^{m+2}), m\geqslant 0, or$\\
 (2)$(xy^m+ay^{m+1})$($(xy^m+ay^{m+2})$) for $m\geqslant 0$.  
\end{theorem}
\begin{proof} 
(1)Let us consider ideals generated by  two elements of type I.
Let  the ideal I be $(xy^m,y^n).$  If $n\leqslant m$,  $I=(y^n)$.  If $n\geqslant m+3$, then $xy^m\in I$ and $x^2y^m=y^{m+3}$ in R. Thus $I=(xy^m).$  Otherwise,  I=$(xy^m,y^n)$ for $n=m+1\; or\; n=m+2.$ 
Let I be $(x,xy^m)$ and then it is same as (x). Let us consider ideals generated by three elements of type I. Let I be $(y^m,xy^n,x)$ which is same as $I=(y^m,x)$. This case have been done above. we make an ideal generated by 4 or more elements of type I is going to be the case we have done above.\\
(2)Let us consider ideals generated by an element of type I and an element of type II.
Note that $y^{m+2}\in (xy^m+ay^{m+1})$ since $xy^m=-ay^{m+1}\in R/I$ and then $y^{m+3}=x^2y^m=-axy^{m+1}=-a^2y^{m+2}\in R/I$ which implies $y^{m+2}\in I$. If $I=(xy^m+ay^{m+1},y^{m_1}), m_1\leqslant m+1$, then $I=(xy^m, y^{m_1})$. For  $m_1>m+1$,  then $I=(xy^m+ay^{m+1})$ since $y^{m+2}\in I$ implies that $y^{m_1}\in I.$  If $I=(xy^m+ay^{m+1}, x)$, then $I=(x,y^{m+1})$. If $I=(xy^m+ay^{m+1},xy^{m_1})$, then for $m_1\leqslant m$, $I=(xy^{m_1},y^{m+1})$. If $m_1>m$, then $y^{m+2}\in I$ implies $xy^{m_1}\in (xy^m+ay^{m+1})$ and thus $I=(xy^m+ay^{m+1})$.  Note that  if $y^{m+3}\in(xy^m+ay^{m+2})$.  Similarly, we can show $(xy^m+ay^{m+2}, y^{m_1})$ is same as one of ideals $(xy^m, y^{m+1}) or (xy^m+ay^{m+2})$.  Let $I$ be $(xy^m+ay^{m+2}, xy^{m_1})$. Then  the ideal $I$ is same as one of $(xy^{m_1}, y^{m+2}) or (xy^m+ay^{m+2})$. If  $I=(xy^m+ay^{m+1}, f, g)$  f,g are type I, then we can easily check  the ideal is same as the one we have already done.\\
(3) Let us I be generated by elements of type II.
 Let $I$ be $(xy^m+ay^{m+1}, xy^{m_1}+by^{m_1+1})$ with $m\leqslant m_1.$  If $a=b$, then $ I=(xy^m+ay^{m+1})$.  If $a\neq b$, then $-by^{m_1+1}=xy^{m_1}=xy^my^{m_1-m}=-ay^{m+1}y^{m_1-m}=-ay^{m_1+1}\in R/I$. That implies $-ay^{m_1+1}+by^{m_1+1}\in I$. Hence $y^{m_1+1}\in I$. Thus $I=(xy^m+ay^{m+1}, xy^{m_1},y^{m_1+1})$ which we already done in (2). Let I be $(xy^m+ay^{m+1},xy^{m_1}+by^{m_1+2})$ with $m\leqslant m_1$,  then $y^{m_1+2}\in I$ since $y^{m+2}\in (xy^m+ay^{m+1})$. Thus $I=(xy^m+ay^{m+1},xy^{m_1},y^{m_1+2})$ which we already done. If $m_1< m$, then $-ay^{m+1}=xy^m=xy^{m_1}y^{m-m_1}=-ay^{m_1+2}y^{m-m_1}=-ay^{m+2}\in R/I$ implies  $y^{m+1}\in I$. That shows $ I=(xy^m,y^{m+1}, xy^{m_1}+ay^{m_1+1})$ which we already done in (2). Let $I$ be $(xy^m+ay^{m+2}, xy^{m_1}+by^{m_1+2})$ with $m\leqslant m_1$.  If $a=b,$ then $I=(xy^m+ay^{m+2}).$ If $a\neq b,\; and\; m=m_1$, then $I=(xy^m, y^{m+2}).$ If $m<m_1$, then $I=(xy^m+ay^{m+2}).$
 \end{proof}
            Let us count the colength of ideal defined by $dim_{\mathbb C}(R/I)$. Denote $\mathbb{N}^0=\mathbb{N}\cup\{0\}$
For computing the colength of ideals let us define partial ordered set $\{(m,n)| m,n \in \mathbb{N}^0, n>0\}$  Define $(m_1,n_1)\geqslant (m_2,n_2)$ of length 1 if $m_1+n_1-1=m_2+n_2,\;m_2\leqslant m_1$ and $n_2\leqslant n_1$.
Let us define the length of $(m,n)>(l,s)$ as $m+n-(l+s)-1.$
We want to look at the ideals. Thus we need some inequalities explaining the properties of the ideals.
Define the length of $I\subset J$ as the $max\{n|I_0=I\subset I_1\subset I_2......\subset I_{n-1}\subset I_n=J\}$ with strict inclusion. Let I be an ideal $(xy^m,y^{m+1})((xy^m,y^{m+2}))$ and J be an ideal $(xy^{m_1},y^{m_2})((xy^{m_1},y^{m_2}))$ where $m\geqslant 0$. We want I and J to have two minimal generators.
 Then we may assume $m_1<m_2$ and $m_1+3 >m_2$.If  $m_1\geqslant m_2$,  then $J=(y^{m_2})$.  Also if $m_1+3\leqslant m_2$, then $J=(xy^{m_1})$ since  $x(xy^{m_1})\in J$ implies that $y^{m_1+3}\in J$. If we have the conditions $m_1\leqslant m$, $m_2\leqslant m+1(m_2\leqslant m+2)$ and $m_1+m_2=m+(m+1)-1(m_1+m_2=m+(m+2)-1)$, then obviously $I\subset J$ and ,in particular, the third condition,$m_1+m_2=m+(m+1)-1(m_1+m_2=m+(m+2)-1)$, gives us the length of $I\subset J$ is 1 we are going to prove in the Lemma 4.  Let  $(m,m+1)>(m_1,m_2)$ be partial ordered pairs of length one with inequalities, $m_1\leqslant m, m_2\leqslant m+1,\; m_1+3>m_2,\; m_1+m_2=2m\; and \;m_1<m_2$. Then $(m_1,m_2)=(m-1, m+1)$  Let $(m,m+2)>(m_1,m_2)$ be partial ordered pairs with inequalities, $3+m_1>m_2,\; m_1<m,\;m_1\leqslant m+2,\;m_1<m_2, \;and\; m_1+m_2=2m+1$. Then $(m_1,m_2)=(m, m+1).$ It is possible that both partial ordered pairs stop at (0,1) which is correspondent to (x,y). The length of the pair $(m,m+1)$ will be 2m since we have $(m,m+1)>.......>(0,1)$ which can be counted by  2m+1-1. The length of $(m,m+2)>......>(0,1)$ is 2m+2-1. That showed $(xy^m,y^{m+1})$ is colength 2m+1 and $(xy^m,y^{m+2})$ is colength 2m+2. Thus $(xy^m+ay^{m+1})$ is colength of 2m+2 since $(xy^m+ay^{m+1})\subseteq(xy^m,y^{m+1})$ is one and $(xy^m+ay^{m+2})$ is colength of 2m+3 since $(xy^m+ay^{m+2})\subseteq(xy^m,y^{m+2})$ is one.(there are no proper ideals between these ideals)  Finally, let us show that we have partial ordered set of length one, then that implies as an ideal is also length one.
\begin{remark}
(1)$(xy^m,y^{m+1})\subset (xy^{m-1},y^{m+1})((xy^m,y^{m+2})\subset (xy^{m},y^{m+1}))$ is length one.
 (2)$(xy^m+ay^{m+1})\subset (xy^m, y^{m+1})((xy^m+ay^{m+2})\subset(xy^m,y^{m+2})), a\neq0, m\geqslant 0$ is length one.  
(3) $(y^m)\subset(xy^{m-1}, y^m )$ is length one, $(xy^m)\subset(xy^m,y^{m+2})$ is length one and $(x)\subset(x,y^2)$ is length one.
\end{remark} 
\begin{Cor}
For nonzeros $a\neq b\in \mathbb{C}$, $(xy^m+ay^{m+1})\neq (xy^m+by^{m+1})$
\end{Cor}
Suppose not, let $I=(xy^m+ay^{m+1})= (xy^m+by^{m+1})$.
Then $xy^m\in I$ which implies $I=(xy^m,y^{m+1})$. However we already showed $(xy^m,y^{m+1})$ is colength 2m+1 which is different from the colength of  $(xy^m+ay^{m+1}).$ Similarly, we can show for $(xy^m+ay^{m+i}) m\geqslant 0, i=0 ,i=1.$ That is for a different nonzero a, we have a different ideal. 
\begin{Cor}
Every ideal in R can have at most 2 generators with the colegth as listed below
$$\left.\begin{array}{c|c}colength & ideal \\\hline 2m+2 & (xy^m+ay^{m+1}),(xy^{m},y^{m+2}),(y^{m+1})\\\hline2m+3 & (xy^m+ay^{m+2}),(xy^{m}),(xy^{m+1},y^{m+2})\\ \end{array}\right.
$$ for $m\geqslant 0$.
\end{Cor}
 Note that $I_a=(xy^m+ay^{m+1})$ contains $y^{m+2}$ and $xy^{m+1}\in I$. Since $x^2y^m+axy^{m+1}\in I$,  $y^{m+3}+axy^{m+1}=y^{m+3}+axy^my=y^{m+3}+a(-ay^{m+1})y=0$ in $R/I_a.$ Thus we have $y^{m+3}-a^2y^{m+2}\in I$ which implies  $y^{m+2}\in I_a.$  Since $x^2y^m+axy^{m+1}=y^{m+3}+axy^{m+1}\in I_a$ and $xy^{m+2}+ay^{m+3}\in I$, \; $xy^{m+1}\in I_a$.
Similarly,  note that $I_a=(xy^m+ay^{m+2})$ contains $y^{m+3}$ and $xy^{m+1}.$
Since $x^2y^m+axy^{m+2}\in I_a$\; $x^2=y^3$ and $xy^m+ay^{m+2}=0$ in $R/I_a,$
$x^2y^m+axy^{m+2}=y^{m+3}-a^2y^{m+4}=0$ in $R/I_a$. Therefore
$y^{m+3}\in I_a.$\; Since $y^{m+3}+axy^{m+2}\in I$ and $xy^{m+1}+ay^{m+3}\in I$, \; $xy^{m+1}\in I_a.$

\begin{lemma}\begin{align}
&lim_{a\rightarrow 0}(xy^{m}+ay^{m+1})=(xy^{m},y^{m+2})\notag\\
&lim_{a\rightarrow 0}(xy^{m}+ay^{m+2})=(xy^{m})\; same\; as\; (xy^m,y^{m+3})\notag\\
&lim_{a\rightarrow \infty}(xy^{m}+ay^{m+1})=(y^{m+1})\notag\\
&lim_{a\rightarrow \infty}(xy^{m}+ay^{m+2})=(xy^{m+1},y^{m+2})\notag
\end{align}\end{lemma}
\begin{proof}
$lim_{a\rightarrow 0}(xy^m+ay^{m+1})$ contains $xy^m$ and since Hilbert scheme is complete, the limit is same as the colength of $(xy^m+ay^{m+1})$. Thus Thm 3 and from lemma 4 implies that the limit is an ideal $(xy^m, y^{m+2}).$ Let $b=1/a$ and then  $lim_{a\rightarrow \infty}(xy^m+ay^{m+1})=lim_{b\rightarrow 0}(bxy^m+y^{m+1})$ which contains $y^{m+1}$. Thus the limit is the ideal $(y^{m+1}).$
Similarly,  we can get the limits for the other cases.
\end{proof}
\subsection {Puctual \text{Hilb}ert scheme of points on the cusp curve is isomorphic to $P^1$ as a scheme} 

Let $R$ be $\mathbb{C}[[x,y]]/(x^2-y^3)$ and  $\x=Spec R.$  The tangent space at the point $I\in Hilb_m(R)$ denoted by $T_IHilb_m(R)$ is known to be same as $Hom_R(I,R/I)$ as a $\mathbb{C}$ module and also let us remark $Hilb_m(\mathbb{C}[x,y]/(x^2-y^3))_{(x,y)})$ is same as $Hilb_m(R).$

 \begin{theorem}
$\text{Hilb}^0_m(X)$ is isomorphic to $\mathbb{P}^1$ as a subscheme of $\text{Hilb}_m(X).$
\end{theorem}
\begin{proof}
Let $I$ be  an ideal $(xy^m+ay^{m+1})\;m\geqslant 0$ parameterized by a. Let us take a look around  the point a=0.  Let $h_{2m+2}:I\rightarrow R/I$ be a ring homomorphism defined by $h_{2m+2}((xy^m+ay^{m+1}))=y^{m+1}$.
 There is an inclusion map(as a set) $\phi:\mathbb{A}^1\rightarrow \text{Hilb}_{2m+2}(X)$ where $\mathbb{A}^1$ is parametrized by a and then $h_{2m+2}=d\phi_{a}(\partial/\partial_a)$.  That is, $D_a(I)=\frac{\partial}{\partial a}(I)=y^{m+1}\neq 0\; in \;R/I$ which gives a nonzero tangent vector in $\text{Hilb}_2(R).$ (As $a\rightarrow \infty$, then we can consider  $(xy^m+1/by^{m+1})$ with $b\rightarrow 0$ which is same as $(bxy^m+y^{m+1})$.)
Around $\infty$, let $I=(bxy^m+y^{m+1})\;m\geqslant 0$.
 Let $j_{2m+2}:I\rightarrow R/I$ defined by $j_{2m+2}((bxy^m+y^{m+1}))=xy^{m}$.
 There is an one to one map(as a set) $\phi:\mathbb{A}^1\rightarrow \text{Hilb}_{2m+2}(X)$ where $\mathbb{A}^1$ is parametrized by b. 
 Since $j_{2m+2}=d\phi_{I_b}(\partial/\partial_b)$,  $D_b(I)=xy^m\neq 0\; in\; R/I$ which gives  a nonzero tangent vector in $\text{Hilb}_2(R)$
  That is, that is a submanifold(or scheme) isomorphic to $\mathbb{A}^1$ in $\text{Hilb}_{2m+2}(X)$.\\
Similarly, we can show that for the ideal $(xy^{m+1}+ay^{m+2})$ and $(bxy^{m+1}+y^{m+2})$.
\end{proof}
\subsection{Tangent space of $\text{Hilb}_m(R)$ along $\text{Hilb}_m^{0}(R)$} 

 \begin{theorem}
 For $I=(xy^m,y^{m+2})\in \text{Hilb}_{2m+2}(X),m\geqslant 0$,  $Hom_{R}(I,R/I)$ is 2m+3 dimensional vector space over $\mathbb{C}.$\; For $I=(xy^{m+1},y^{m+2})\in \text{Hilb}_{2m+3}(X), m\geqslant 0$, $Hom_{R}(I,R/I)$ is 2m+4 dimensional vector space over $\mathbb{C}$.
\end{theorem}
\begin{proof}
A free resolution of an ideal $I=(xy^{m},y^{m+2})$ is 
  $$\begin{CD}
R\oplus R @>\partial _1>> R\oplus R@>\partial_0>>  I=(xy^{m},y^{m+2})@>>>0
\end{CD}$$   
where $\partial_1=:     \left[\begin{array}{cc}y^2 &  x \\ x& y\end{array}\right] and\; \partial_0=:[xy^{m},-y^{m+2}],m\geqslant 0.$
Let us take $Hom_R(\bullet ,R/I)$ to the exact sequence, then we get
 $$\begin{CD}
0@>>>Hom(I,R/I) @>\bar{\partial_0}>> Hom(R\oplus R,R/I)@>\bar{\partial_1}>>  Hom_R(R\oplus R,R/I)
\end{CD}$$   
where $\partial_1=:     \left[\begin{array}{cc}y^2 &  x \\ x& y\end{array}\right] and\; \partial_0=:[xy^{m},-y^{m+2}],m\geqslant 0$
$Ker\bar{\partial_1}=\{(h_1,h_2)|h_1x+h_2y=0, h_1y^2+h_2x=0\in R/I\}.$
Since a basis B for R/I is $\{1,x,y,...,y^{m+1},xy, xy^2,...,xy^{m-1}\}$ and $h_1,h_2\in I$ implies that $(h_1,h_2)\in ker \bar{\partial_1}$, it is enough to check for $h_i\in B, i=1, 2.$ \\It is clear that $K_1=:<(0,y^{m+1}),(0,xy^{m-1}),(y^{m+1},0),(y^m,0),(xy^{m-1},0)>\subseteq Ker\bar{\partial_1}.$  Since $x^2-y^3=0$,  $\;(x,-y)\in Ker\bar{\partial_1}$ and $(y^2,-x)\in Ker\bar{\partial_1}$ which implies  $K_2=:\{f(x,-y),\;g(y^2,-x)|
f,g\in R/I\}\subseteq Ker\bar{\partial_1}$. 

$$\left.\begin{array}{|c|c|c|}\hline f,g\in B& (fy,-fx) & (gx,-gy^2) \\\hline 1 & (y,-x) & (x,-y^2) \\\hline x & (xy,-y^3) & (y^3,-xy^2) \\\hline y & (y^2,-xy) & (xy,-y^3)\\\hline y^2 & (y^3,-xy^2) & (xy^2,-y^4) \\\hline : & : & : \\\hline y^{m-1} & (y^{m},-xy^{m-1}) & (xy^{m-1},-y^{m+1})\\\hline y^m & (y^{m+1},-xy^m) & (0,0) \\\hline y^{m+1} & (0,0) & (0,0) \\\hline xy & (xy^2,-y^4) & (y^4,-xy^3) \\\hline 
xy^2&(xy^3,-y^5) & (y^5,-xy^4)\\\hline:&:&: \\\hline xy^{m-2}&(xy^{m-1},-y^{m+1})&(y^{m+1},-xy^{m})\\\hline xy^{m-1}&(0,0)&(0,0) \\\hline
\end{array}\right.$$
Delete one of two which are same or belongs to $K_1$.
$$\left.\begin{array}{|c|c|c|}\hline f,g\in B& (fy,-fx) & (gx,-gy^2) \\\hline 1 & (y,-x) & (x,-y^2) \\\hline x & (xy,-y^3) &  \\\hline y & (y^2,-xy) & \\\hline y^2 & (y^3,-xy^2) &  \\\hline : & : & : \\\hline y^{m-2}&(y^{m-1},-xy^{m-2})&\\\hline y^{m-1} && \\\hline y^m &  & (0,0) \\\hline y^{m+1} & (0,0) & (0,0) \\\hline xy & (xy^2,-y^4) &\\\hline 
xy^2&(xy^3,-y^5) & \\\hline:&:&: \\\hline xy^{m-3}&(xy^{m-2},-y^m)&\\\hline xy^{m-2}&&\\\hline xy^{m-1}&(0,0)&(0,0) \\\hline
\end{array}\right.$$

Let us  $K_1\cup K_2$ to be a R-module generated by $K_1$ and $K_2$ is same as $Ker \partial_1.$ $|\bullet|$ be  the number of basis of $\bullet $ over  $\mathbb{C}$. Since $|K_1|=5  , |K_2|=2m+1 $ and $\{(y^{m+1},-xy^m),(y^m,-xy^{m-1}),\\(xy^{m-1},-y^{m+1})\}$ belongs to $ K_1$ as well.
Thus $2m+1+5-3=2m+3.$\;We can see $|K_1\cup K_2|=2m+3.$\; We can explicitly find out a basis for $Ker\bar{\partial_1}$ over $\mathbb{C}$:
$\{(0,y^{m+1}),(0,xy^{m-1}),(y^{m+1},0),\\(y^m,0),(xy^{m-1},0), (x,-y^2), (xy,-y^3),..,(xy^{m-2},-y^m),(y,-x),(y^2,-xy),..,(y^{m-1},-xy^{m-2})\}.$
A free resolution of an ideal $I=(xy^{m+1},y^{m+2})$ is 
  $$\begin{CD}
R\oplus R @>\partial _1>> R\oplus R@>\partial_0>>  I=(xy^{m+1},y^{m+2})@>>>0
\end{CD}$$   
where $\partial_1=:     \left[\begin{array}{cc}x &  y \\ y^2& x\end{array}\right]\; and\; \partial_0=:[xy^{m+1},-y^{m+2}],m\geqslant 0.$
Let us take $Hom_R(\bullet ,R/I)$ to the exact sequence, then we get
 $$\begin{CD}
0@>>>Hom_R(I,R/I) @>\bar{\partial_0}>> Hom_R(R\oplus R,R/I)@>\bar{\partial_1}>>  Hom_R(R\oplus R,R/I).
\end{CD}$$   
Since $Ker\phi_1=Hom_R(R,R/I)$ and $Hom_R(R\oplus R, R/I)=R/I\oplus R/I$,  $Ker\phi_1$ is same as $Ker \bar{\partial_1}$
where $\bar{\partial_1}:R/I\oplus R/I \rightarrow R/I\oplus R/I$ 
defined as $\left[\begin{array}{cc}\bar{x} & \bar{y}\\ \bar{y^2}& \bar{x}\end{array}\right], \bar{x},\bar{y}, \bar{y^2}\in R/I$
$Ker\bar{\partial_1}=\{(h_1,h_2)|h_1x+h_2y=0, h_1y^2+h_2x=0\in R/I\}.$ Since a basis B for R/I is $\{1,x,y,...,y^{m+1},xy, xy^2,...,xy^m\}$ and $h_1,h_2\in I$ implies $(h_1,h_2)\in\bar{\partial_1}$, it is enough to check for $h_i\in B, i=1,\; 2.$  It is clear that  $<(0,y^{m+1}),(0,xy^m),(y^{m+1},0),(xy^m,0),(xy^{m-1},0)>\subseteq Ker\bar{\partial_1}.$  Since $x^2-y^3=0$,  $\;(x,-y^2)\in Ker\bar{\partial_1}$ and $(y,-x)\in Ker\bar{\partial_1}$ which implies  $\{f(x,-y^2),\;g(y,-x)| f,g\in R/I\}\subseteq Ker\bar{\partial_1}$. In a similar way, 
we can explicitly find out a basis for $Ker\bar{\partial_1}$ over $\mathbb{C}$:
$\{(0,y^{m+1}),(0,xy^m),(y^{m+1},0),(xy^m,0),(xy^{m-1},0), 
(x,-y^2),(xy,-y^3),\\..., (xy^{m-2},-y^m), (y,-x),(y^2,-xy),..., 
,...,(y^m,-xy^{m-1})\}.$
\end{proof}

\begin{remark}  Since  an ideal  $I=(xy^m+ay^{m+1})((xy^m+ay^{m+2}))$ is generated by one element,
$Hom_R(I,R/I)$ is same as homomorphisms from $xy^m+ay^{m+1}$ to an element in $R/I.$ Thus, it is clear $T_IHilb_{2m+2}(X)$ is 2m+2(2m+3)-dimensional vector space. For example, for an ideal $I=(x+ay)$, the tangent space of Hilbert scheme at the point I is $T_IHilb_2(X)=\{x+ay+a_1+a_2y\},\;a_1,\;a_2\in\mathbb{C}.$  If  a approaches to 0, the limit ideal $I_0$ is $(x,y^2)$ and $T_{I_0}Hilb_2(X)=\{x+a_1+a_2y, y^2+a_3y\},\;a_1,\;a_2,\;a_3\in\mathbb{C}.$  If a approaches to $\infty$, we can let $b=1/a$  and  $b\rightarrow 0$ and then I=(bx+y) and $T_IHilb_2(X)=\{bx+y+b_1+b_2y\},\;b_1,\;b_2\in\mathbb{C}.$ The limit ideal $I_{\infty}$ is $(y)$ and  $T_{I_{\infty}}Hilb_2(\x)=\{ y+b_1+b_2x\},\;b_1,\;b_2\in\mathbb{C}.$  Thus $Hilb_2(\x)$ has only one  singularity at $a=0\in \mathbb{P^1}$ where $\mathbb{P^1}$ is isomorphic to the punctual one  $Hilb_2^0(\x)$. Similarly, we can see $Hilb_3(X$ has only one singularity at $a=\infty$  along $Hilb_3^0(\x)$ where $(x+ay^2)$ is parameterizing $Hilb_3^0(X).$ 
\end{remark}
   \newpage
\bibliographystyle{aip}
 
\bibliography{uctest}

\end{document}